\documentclass[12pt,reqno]{amsart}

\usepackage{amssymb,amsmath,amsthm,amscd}
\usepackage{color,cancel,todonotes,hyperref}
\usepackage[english]{babel}
\usepackage[utf8]{inputenc}      % accents dans le source
\usepackage[T1]{fontenc}          % accents dans le pdf

\DeclareMathAlphabet{\mathpzc}{OT1}{pzc}{m}{it}

\usepackage{changepage, array}
\usepackage{graphicx}
\usepackage{comment}

\author{Roberto Bramati}
\address{Department of Mathematics: Analysis, Logic and Discrete Mathematics, Ghent University, Krijgslaan 281, 9000 Ghent, Belgium}
\email{Roberto.Bramati@UGent.be}
\author{Angela Pasquale}
\address{Universit\'e de Lorraine, CNRS, IECL, F-57000 Metz, France}
\email{angela.pasquale@univ-lorraine.fr}
\author{Tomasz Przebinda}
\address{Department of Mathematics, University of Oklahoma, Norman, OK 73019, USA}
\email{tprzebinda@ou.edu}

\thanks{Part of this research took place within the online Research Community on Representation Theory and Noncommutative Geometry sponsored by the American Institute of Mathematics. The first author is supported by the FWO Odysseus 1 grant G.0H94.18N: Analysis and Partial Differential Equations and by the Methusalem programme of the Ghent University Special Research Fund (BOF) (Grant number 01M01021). The second author would like to thank the University of Oklahoma for hospitality and financial support.
The third author gratefully acknowledges hospitality and financial support from the Universit\'e de Lorraine and partial support from the National Science Foundation under Grant DMS-2225892.}

\title[The resonances of the Capelli operators]{The resonances of the Capelli operators for small split orthosymplectic dual pairs}
\def\g{\mathfrak g}

\def\y{\mathrm y}

\def\sp{\mathfrak {sp}}

\def\o{\mathfrak o}

\def\u{\mathfrak u}

\newcommand{\R}{\mathbb{R}}
\newcommand{\C}{\mathbb{C}}
\newcommand{\DD}{\mathbb{D}}

\newcommand{\ZZ}{\mathbb{Z}}

\def\m{\mathfrak m}

\def\ss1{\mathfrak s_{\overline 1}}

\def\hs1{\mathfrak h_{\overline 1}}

\def\supp{\mathrm{supp}}

\def\Op{\mathrm{Op}}
\def\Ker{\mathrm{Ker}}

\def\G{\mathrm{G}}
\def\N{\mathrm{N}}

\def\E{\mathrm{E}}

\def\SOg{\mathrm{S}\mathrm{O}}
\def\K{\mathrm{K}}
\def\H{\mathrm{H}}
\def\M{\mathrm{M}}

\def\Sg{\mathrm{S}}

\def\L{\mathrm{L}}
\def\Bbb{\mathbb}

\def\N{\mathrm{N}}
\def\A{\mathrm{A}}
\def\H{\mathrm{H}}

\def\GL{\mathrm{GL}}
\def\SL{\mathrm{SL}}
\def\SO{\mathrm{SO}}

\def\Sp{\mathrm{Sp}}

\def\Og{\mathrm{O}}
\def\Ug{\mathrm{U}}

\def\Ind{\mathrm{Ind}}

\def\Ca{\mathcal{C}}

\def \t{\tilde}
\def \wt{\widetilde}

\def\W{\mathsf{W}}
\def\Wv{\mathsf{W}}

\def\Uv{\mathsf{U}}
\def\V{\mathsf{V}}

\def\Dsf{\mathsf{D}}
\def\X{\mathsf{X}}
\def\Y{\mathsf{Y}}
\def\Xv{\mathrm{X}}
\def\Yv{\mathrm{Y}}

\def\proj{\mathsf{P}}
\def\rad{\mathsf{R}}

\def\Zb{\mathbb {Z}}
\def\End{\mathop{{\rm End}}\nolimits}

\def\det{\mathop{{\rm det}}\nolimits}
\def\ad{\mathop{{\rm ad}}\nolimits}

\def\Hom{\mathop{{\rm Hom}}\nolimits}
\def\Re{\mathrm{Re}}
\def\Im{\mathop{{\rm Im}}\nolimits}

\def\tr{\mathop{{\rm tr}}\nolimits}

\def\sign{\mathop{{\rm sign}}\nolimits}

\def\lim{\mathop{{\rm lim}}\nolimits}

\newcommand\inner[2]{\langle #1,#2\rangle}

\def\Res{\mathop{{\rm Res}}\limits}

\def\Ind{\mathop{{\rm Ind}}\nolimits}
\def\supp{\mathop{{\rm supp}}\nolimits}

\def\U{\mathcal{U}}

\def\Z{\mathcal{Z}}

\def\Ss{\mathcal{S}}

\newcounter{thh}
\newtheorem{thm}[thh]{{Theorem}}

\newtheorem{pro}[thh]{{Proposition}}
\newtheorem*{pro*}{{Proposition}}
\newtheorem{cor}[thh]{{Corollary}}
\newtheorem*{coro*}{{Corollary}}
\newtheorem{lem}[thh]{{Lemma}}

\newcounter{remcounter}
{\theoremstyle{definition}\newtheorem{rem}[remcounter]{{Remark}}}
{\theoremstyle{definition}\newtheorem{rems}[remcounter]{{Remarks}}}

\theoremstyle{definition}
\newtheorem*{defi*}{{D\Ca\Cafinition}}

\newtheorem*{nota*}{{Notation}}
\newenvironment{prf}{\begin{proof}}{\end{proof}}

\def\muet{ \ifthenelse{\equal{a}{b}}}

\def\Z'{\Bbb{Z}'}
\def\fo{\mathcal{F}}

\def\biblio{\sloppy
\bibliographystyle{alpha}
\bibliography{CapelliResonance-arXiv2}}

\begin{document}
\date{}

\begin{abstract}
Let $(\G,\G’)$ be a reductive dual pair in ${\rm Sp}(\Wv)$ with ${\rm rank}\, \G \leq {\rm rank}\, \G’$ and $\G'$ semisimple. The image of the Casimir element of the universal enveloping algebra of $\G'$ under the Weil representation $\omega$ is a Capelli operator. It is a hermitian operator acting on the smooth vectors of the representation space of $\omega$. We compute the resonances of a natural multiple of a translation of this operator for small split orthosymplectic dual pairs. The corresponding resonance representations turn out to be $\G\G’$-modules in Howe's correspondence. We determine them explicitly. 
\end{abstract}

\maketitle

\tableofcontents

\def\v0{v_0}
\def\u0{u_0}
\def\sm{\mathrm {sum}}

%%%%%%%%%%%%%%%%%%%%%%%%%
\section{Introduction}

The notion of resonances originated in the '30s in Quantum Mechanics. 
As described in \cite{Harrell}, the story goes back to 1926, when Schr\"odinger studied the Stark effect, i.e. the shifts caused to hydrogen's emission spectrum by the application of a constant field. The hydrogen Stark Hamiltonian is the unbounded operator on $\L^2(\R^3)$ given by
\[H=\Delta-\frac{1}{|x|}+\kappa x_1\,\]
where $\kappa\geq 0$ is the electrical field strength and the field acts in the $x_1$-direction. 
In Schr\"odinger's model, the energies were the eigenvalues of $H$ and the model was based on eigenfunction expansions. This work was received with great enthusiasm by many physicists of the time. For example, Epstein's 1926 article in Nature, see \cite{Epstein}, considered it to be ``of extraordinary importance''. It had great influence on modern physics. Nevertheless, Schr\"odinger's analysis contained a mistake: the hydrogen Stark Hamiltonian has no eigenvalues if $\kappa>0$. 
This absence of eigenvalues was first noticed by Oppenheimer in 1928. Oppenheimer did not prove it, but referred for the proof to a work of Weyl, where there was no proof either.
Finally, in 1951, Titchmarsh proved that the Stark Hamiltonian has no eigenvalues.
The ``phantom eigenvalues'' in the Stark effect are in fact resonances and
the ``eigenfunction expansions'' are resonant state expansions. 
Resonances are discrete spectral data, which might replace eigenvalues for differential operators with a continuous spectrum. 

Rigorous mathematical approaches to resonances were formulated only in the '70s and '80s. 
Consider for example a Schr\"odinger operator
$H=\Delta+V$ on $\L^2(\R^n)$. Here
$\Delta=-\sum_{j=1}^n \frac{\partial^2}{\partial x_j^2}$ and $V$ is a potential acting as a multiplication operator.
Under suitable assumptions, $H$ is an unbounded self-adjoint operator on $\L^2(\R^n)$
with continuous spectrum $[0,+\infty)$.
For $\zeta \in \C \setminus [0,+\infty[$, the resolvent of $H$, i.e. 
$R_H(\zeta)=(H-\zeta)^{-1}$
is a bounded operator on $\L^2(\R^n)$, depending holomorphically on $\zeta$.
As an operator on  $\L^2(\R^n)$, $R_H(\zeta)$ has no analytic extension across the spectrum of $H$.
But we can replace $\L^2(\R^n)$ by a smaller dense subspace, like $C^\infty_c(\R^n)$
and consider the map
\[
\C \setminus [0,+\infty) \; \ni \zeta \longrightarrow 
R_H(\zeta)=(H-\zeta)^{-1} \in {\rm Hom}(C^\infty_c(\R^n),{C^\infty_c}(\R^n)^*)\,,
\]
which might have some continuation across $ [0,+\infty)$, possibly to a Riemann surface.
If the continuation is meromorphic, then the poles are called the resonances of $H$.

It turns out to be convenient to replace the variable $\zeta \in \C \setminus [0,+\infty)$ with 
$z\in \C^+=\{w  \in \C: \Im w>0\}$ by substituting $\zeta=z^2$.
Define 
\[
R(z)=R_{H}(z^2)=(H-z^2)^{-1}\,.
\]
The problem of meromorphic extension of $R_H$ as a function of $\zeta\in\C \setminus [0,+\infty)$ is equivalent to that of $R$ as a function of $z\in \C^+$.

The theory of resonances of $H=\Delta+V$ appears naturally in many branches of mathematics, physics and engineering. We refer to \cite{DZ19} for more information.

The study of resonances of differential operators was extended beyond Euclidean settings. 
The most investigated situations concern the Laplacian on a complete noncompact Riemannian manifold with bounded geometry, such as hyperbolic and asymptotically hyperbolic manifolds, symmetric or locally symmetric spaces (mostly of rank 1). This is motivated by applications to geometric scattering, spectral theory, trace formulas, PDE's, and dynamical systems.

Riemannian symmetric spaces of the noncompact type are attractive because they have 
a well understood geometry, a well developed Fourier analysis (the Helgason-Fourier transform) and allow using tools from representation theory.
Recall that such a space is of the form $\G/\K$, where 
$\G$ is a connected noncompact real semisimple Lie group with finite center and $\K$ is a maximal compact subgroup of $\G$. 
The left-regular representation $\L$ of $\G$ on $\L^2(\G/\K)$ decomposes into isotypic components according to 
\[
\L^2(\G/\K)=\int_{\mathfrak{a}^*} \L^2(\G/\K)_{\pi_{i\lambda}} \; \frac{d\lambda}{c(i\lambda)c(-i\lambda)}
\]
where $\mathfrak{a}$ is Cartan subspace of the Lie algebra $\mathfrak{g}$ of $\G$, $\pi_{i\lambda}$ is the unitary principal series representation of $\G$ of parameter $\lambda\in \mathfrak{a}^*$ 
and $c(i\lambda)$ is Harish-Chandra's $c$-function. 
This decomposition is realized via the Helgason-Fourier inversion formula:
\begin{equation}
\label{decHelgason}
f(x)=\int_{\mathfrak{a}^*} \underbrace{(f\times \varphi_{i\lambda})(x)}_{f_{\pi_{i\lambda}}(x)} \, \frac{d\lambda}{c(i\lambda)c(-i\lambda)} \qquad (f\in C_c^\infty(\G/\K))\,,
\end{equation}
where $\varphi_{i\lambda}$ is the spherical function of spectral parameter $i\lambda$ and 
$\times$ denotes the convolution of functions on $\G/\K$.
Let $\mathcal{U}(\mathfrak{g})$ be the enveloping algebra of $\g$, $\mathcal{U}(\mathfrak{g})^\G$ the subalgebra of $\G$-invariant elements and let 
$\mathcal{C}\in \mathcal{U}(\mathfrak{g})^\G$  be the Casimir element. 
Then $\Delta=\L(-\mathcal{C})$ is the (positive) Laplacian on $\G/\K$. It is an
essentially self-adjoint unbounded operator on $\L^2(\G/\K)$, with continuous spectrum 
$[\rho_X^2,+\infty[$, where $\rho_X^2$ is a positive constant. 
It acts by the scalar $\inner{\lambda}{\lambda}+\rho_X^2$ on $\L^2(\G/\K)_{\pi_{i\lambda}}$.
Hence $\Delta-\rho_X^2$ has continuous spectrum $[0,+\infty)$
and acts by the scalar $\inner{\lambda}{\lambda}$ on $\L^2(\G/\K)_{\pi_{i\lambda}}$.
The resolvent 
$R(z)=(\Delta-\rho_X^2-z^2)^{-1}$ is a holomorphic function of $z\in \C^+$,
with values in the space of bounded operators on $\L^2(\G/\K)$.
It extends meromorphically from $\C^+$ to $\C$ (or to a Riemann surface over $\C$) by considering it as an operator on $C_c^\infty(\G/\K)$.
The explicit decomposition of $\L$ in \eqref{decHelgason} yields for $z\in \C^+$ and $f\in C_c^\infty(\G/\K)$:
\[
R(z)f(x)=\int_{\mathfrak{a}^*} \frac{1}{\inner{\lambda}{\lambda}-z^2} \; f_{\pi_{i\lambda}}(x) \, \frac{d\lambda}{c(i\lambda)c(-i\lambda)}\,.
\]
A meromorphic extension of $R(z)$ may exist because the terms in the integrand admit a meromorphic extension in $\lambda\in \mathfrak{a}_\C^*$. Namely:
\begin{enumerate}
\item
$\{\pi_{i\lambda}; \lambda\in \mathfrak{a}^*\} \subset \{ \text{spherical principal series representations $\pi_\lambda$}, \lambda\in \mathfrak{a}_\C^* \}$;
\item
$f_{\pi_{\lambda}}=f \times \varphi_\lambda$ exists and is a Paley-Wiener type function of $\lambda\in\mathfrak{a}_\C^*$;
\item
the Plancherel density $\frac{1}{c(i\lambda)c(-i\lambda)}$ extends as a meromorphic function in $\mathfrak{a}_\C^*$;
\item
$\frac{1}{\inner{\lambda}{\lambda}-z^2}$ is meromorphic in $\mathfrak{a}_\C^*$.
\end{enumerate}
Suppose that the resolvent has a meromorphic extension to $\C$, as it does in the real rank-one case, see e.g. \cite{MW00, HP09}, and let $z_0$ be a resonance. 
Since the Laplacian is $\G$-invariant, the group $\G$ acts on the residue space 
\[
\left\{ \Res_{z=z_0} R(z)f; f\in C^\infty_c(\G/\K)\right\}
\]
by the left-regular action. This is the resonance representation at $z_0$.

To summarize, for the Laplacian on $\G/\K$ we have:
\begin{enumerate}
\item
a unitary representation $\L$ of a reductive Lie group $\G$,
\item
a differential operator $\L(-\mathcal{C}+\text{constant})$,
where $\mathcal{C}$ is the Casimir element; 
\item
a representation of $\G$ at each resonance of $\L(-\mathcal{C}+\text{constant})$.
\end{enumerate}

It seems natural to replace $\L$ by an arbitrary unitary representation $\omega$ of $\G$ and 
study the resonances and the associated representations for
$
\omega(-\mathcal{C}+\text{constant})\,.
$ 

\medskip

Consider a reductive dual pair $(\G,\G')$ in the sense of Howe (see section \ref{section:Plancherel} for definitions) in a symplectic group $\Sp(\Wv)$. 
Let $\mathcal{U}(\g)$ denote the universal enveloping algebra of $\g$, and similarly for $\g'$. 
Let $\omega$ denote the Weil representation of the metaplectic group $\wt{\Sp}(\Wv)$ corresponding to a fixed unitary character of $\R$. 
Then the $\G$-invariants $\mathcal{U}(\g)^\G$ and $\G'$-invariants $\mathcal{U}(\g')^{\G'}$ are mapped by $\omega$ onto the same algebra of operators:
\begin{equation}
\label{Capelli-identity}
\omega\left(\mathcal{U}(\g)^\G\right)=\omega\left(\mathcal{U}(\g')^{\G'}\right)\,.
\end{equation}
Any operator in this algebra is called a Capelli operator. The equality \eqref{Capelli-identity} is a consequence of \cite[Theorem 7]{HoweRemarks}. See also \cite{PrzebindaInfinitesimal} and \cite[(0.1)]{Itoh05}.

Suppose that $\G'$ is semisimple. Then $\mathcal{U}(\g')^{\G'}$ contains a well-defined Casimir
element $\mathcal{C}'$. 
From the above equality, we know that there is $\mathcal{C}''\in\mathcal{U}(\g)^\G$ such that 
$\omega(\mathcal{C}'')=\omega(\mathcal{C}')$. Moreover, $\mathcal{C}''$ is unique because ${\rm rank} \G\leq {\rm rank} \G'$ (see \cite{PrzebindaInfinitesimal}).  
In the example of $(\Og_{1,1},\Sp_2(\R))$, we have $\mathcal{C}''=h^2-1$, where $h=\begin{pmatrix}
0 & 1 \\ -1 & 0
\end{pmatrix}$ is a basis of $\o_{1,1}$; see \eqref{Capelli1}. For the dual pairs $(\Sp_2(\R),\Og_{p,p})$, we have $\mathcal{C''}=\mathcal{C}-(p-1)^2+1$, where $\mathcal{C}$ is the Casimir operator of 
$\Sp_2(\R)$; see \eqref{Capelli100}.

The Capelli operator $\mathcal{C}^+$ we study in this paper is a natural multiple of a translation of $\omega(\mathcal{C'})=\omega(\mathcal{C''})$. It is chosen so that its continuous spectrum is $[0,+\infty)$. Furthermore, we consider orthosymplectic dual pairs $(\G,\G')$ with the rank of $\G$ or $\G'$ equal to $1$ and the orthogonal group is of the form $\Og_{p,p}$. Our goal is to determine the resonances of the Capelli operator $\mathcal{C}^+$ as an unbounded operator on the Hilbert space of $\omega$. We use the easier of the two groups in the dual pair, which is $\G$ in our notation, to obtain the spectral analysis of the operator $\mathcal{C}^+$, the meromorphic continuation of its resolvent and the resonance representations as $\G$-modules.
We could have tried to do the analysis of resonances of $\omega(\mathcal{C'})$ working with the more difficult group $\G'$. Nevertheless, since $\omega(\mathcal{C}')=\omega(\mathcal{C}'')$ and because of Howe's correspondence, we do not have to do it: the result would be the same. 

This paper is organized as follows. In section \ref{section:abstract-resonances} we outline the general idea of resonances for an operator of the form $\omega(\mathcal{C})$ where $\omega$ is a unitary representation of a Lie group $\E$ and $\mathcal{C}\in \mathcal{U}(\mathfrak{e})^\E$, where $\mathfrak{e}$ is the Lie algebra of $\E$. In section \ref{sectionO11-Sp2R}, we provide a complete
analysis of the resonances and residue representations for the Capelli operator $\mathcal{C}^+$ for the dual pair $(\Og_{1,1},\Sp_{2}(\R))$. The case $(\Og_{1,1},\Sp_{2n}(\R))$ with $n>1$ could be treated in a similar way, but the description of the resonance representations would be less explicit. In section \ref{section:Plancherel} we show how to decompose the restriction of the Weil representation to the smaller member of an orthosymplectic dual pair in the stable range. Finally, in the last section we apply the results of section \ref{section:Plancherel} to the dual pair 
$(\Sp_{2}(\R),\Og_{p,p})$ with $p\geq2$ and study the resonances and the associated resonance representations of $\mathcal{C}^+$. In Appendices \ref{appendix:Weil representation} and \ref{appendix:O11Sp2}  we recall some facts about the Weil representation.

\section{Abstract resonances}
\label{section:abstract-resonances}

Let $\E$ be a real Lie group with Lie algebra $\mathfrak{e}$ and let $(\omega,\V)$ be a unitary representation of $\E$. Let $(\cdot,\cdot)_\V$ denote the inner product of $\V$ and $\|\cdot\|_\V$ the associated norm. We denote by $\V^\infty$ the space of $C^\infty$-vectors for $(\omega,\V)$. It consists of the elements $v\in\V$ for which the map $E\ni g\mapsto \omega(g)v\in \V$ is $C^\infty$.

Let $\U(\mathfrak e)$ denote the universal enveloping algebra of the complexification of $\mathfrak{e}$. For short, the derived representation of $\U(\mathfrak e)$ acting on $\V^\infty$  will be indicated by the same symbol $\omega$ (in place of $d\omega$): 
\begin{equation}
\label{differentiation}
\omega(X)v=\left.\frac{d}{dt} \omega(\exp(tX))v\right|_{t=0} \qquad 
(X\in \mathfrak{e}, \; v\in \V^\infty)\,.
\end{equation}
As shown by Segal in \cite{Segal59}, $\V^\infty$ is the largest subspace of $\V$ on which all the $\omega(u)$, $u\in \U(\mathfrak{e})$, are defined (even if a specific $\omega(u)$ may be extended to a larger domain in $\V$).  The space $\V^\infty$ has a topology defined by the family of seminorms $\left\{p_D; D\in \U(\mathfrak e)\right\}$
where $p_D(v)=\|\omega(D)v\|_\V$ for $v\in \V^\infty$.
Then $\omega$ is a smooth representation of $\E$ on the Fréchet space $\V^\infty$.

For every $u\in \U(\mathfrak{e})$, the operator $\omega(u)$ with domain $\V^\infty$ is closable, with closure denoted by $\overline{\omega(u)}$. Let $\E_0$ denote the identity connected component of $\E$ and let
$\mathsf{D}$ be an $\E_0$-invariant dense subspace of $\V$ contained in $\V^\infty$. Poulsen proved (see \cite[p. 91 and Corollary 1.2]{Poulsen72}) that
\[
    \overline{\omega(u)}=\overline{\omega(u)|_\mathsf{D}}\,.
\]
This is useful because, despite the fact that 
$\V^\infty$ is a natural domain
for the operators $\omega(u)$, for practical purposes, it might be convenient to work on smaller dense domains. The above property says that the choice of the (group invariant) dense domain inside $\V^\infty$ is immaterial.

Let $u\mapsto u^+$ denote the conjugate-linear involution
of $\U(\mathfrak e)$ such that $X^+=-X$ for all $X\in\mathfrak{e}$. An element $u\in \U(\mathfrak e)$ is said to be hermitian if $u^+=u$.
For $u\in \U(\mathfrak{e})$ and $v,v'\in \V^\infty$ we have $(\omega(u)v,v')_\V=(v,\omega(u^+)v')_\V$.
So $\omega(u)^*$ extends $\omega(u^+)$.
 
For $\zeta\in\C$ outside the spectrum $\sigma(\omega(u))$ of $\overline{\omega(u)}$
the operator
\[
R_{\omega(u)}(\zeta)=(\overline{\omega(u)} -\zeta)^{-1}:\V\to\V
\]
exists and is continuous. This is how we understand the resolvent of $\omega(u)$ at $\zeta$. 
It is a holomorphic function on $\C\setminus \sigma(\omega(u))$ with values in the space of bounded linear operators on $\V$.

A decomposition of $\omega$ into unitary representations leads to an explicit expression for $R_{\omega(u)}$. Indeed, suppose that the unitary representation $(\omega,\V)$ of $\E$ decomposes as a direct integral 
\begin{equation}
\label{direct-integral}
\V=\int^\oplus_{\widehat{\E}} \V_\pi \, d\mu(\pi)
\end{equation}
of unitary isotypic representations $\V_\pi$ of type $\pi\in\widehat{\E}$, where the parameter set  $\widehat{\E}$ is the unitary dual of $\E$. (Recall that any unitary representation of a type I group on a separable Hilbert space has an essentially unique direct integral decomposition of the form \eqref{direct-integral}, see e.g. \cite[\S 2.4]{MackeyUnitaryBook}.) Notice that the support of the measure $\mu$ need not be the entire $\widehat{\E}$. 
Thus every element $v\in \V$ is represented by vectors $v_\pi\in \V_\pi$ and we will write this as 
\[
v=\int^\oplus_{\widehat{\E}} v_\pi \, d\mu(\pi)\,.
\]
The inner product on $\V$ is given in terms of the inner products 
$(\cdot,\cdot)_\pi$ on the $\V_\pi$'s by 
\[
(u,v)=\int^\oplus_{\widehat{\E}} (u_\pi,v_\pi)_\pi d\mu(\pi)\,,
\]
and the elements of $\V$ are precisely the measurable vector fields $v:\widehat{\E} \to \prod_{\pi\in \widehat{\E}} \V_\pi$ which are square integrable, i.e. $(v, v)<\infty$.
We identify two fields that are equal almost everywhere. For additional information on direct integrals and linear operators on them, see \cite[Chapitre II, \S 1-3]{Dixmier1969-Hilbert} and \cite{Nussbaum-Duke64}. The action of $\E$ on $\V$ diagonalizes according to:
\[
(\omega(g)v)_\pi=\pi(g)v_\pi \qquad (g\in\E\,,\ \pi\in \widehat{\E})\,.
\]

The following lemma was proved in \cite[Lemma 2]{Arnal76}.
\begin{lem}
\label{Arnal}
Keep the above notation and let $\{X_j\}$ be a basis of $\mathfrak{e}$. Then 
$v=\int^\oplus_{\widehat{\E}} v_\pi \, d\mu(\pi)$ belongs to $\V^\infty$
if and only if the following two conditions are satisfied:
\begin{enumerate}
\item 
$v_\pi\in \V_\pi^\infty$ for almost all $\pi \in \widehat{\E}$;
\item 
the fields $(\pi(X_i)^nv_\pi)$ are square integrable for every integer $n\geq 0$.
\end{enumerate}
In this case, for every $u\in\U(\mathfrak{e})$, we have
\[
\omega(u)v=\int_{\widehat{\E}}^\oplus \pi(u)v_\pi \,d\mu(\pi)\,.
\]
\end{lem}

A short argument based on Lemma \ref{Arnal} and the definitions involved proves the following corollary.
\begin{cor}
\label{cor:Arnal}
Let $(\omega,\V)$ be a unitary representation as above, with isotypic unitary decomposition 
$\V=\int_{\widehat{\E}}^\oplus \V_\pi \, d\mu(\pi)$. Let $u\in \mathcal{U}(\mathfrak{e})$ and let $v=\int_{\widehat{\E}}^\oplus v_\pi \, d\mu(\pi)$ be in the domain of $\overline{\omega(u)}$. 
Then the $v_\pi$ are in the domain of $\overline{\pi(u)}$ for almost all $\pi$ and 
\[
\overline{\omega(u)}v=\int_{\widehat{\E}}^\oplus \overline{\pi(u)}v_\pi \, d\mu(\pi)\,.
\]
\end{cor}

For $\zeta\in\C \setminus \sigma({\omega(u))}$, the operator $\overline{\omega(u)}-\zeta$ is closed and invertible with bounded inverse $R_{\omega(u)}(\zeta)$. By \cite[Theorem 3(2)]{Nussbaum-Duke64} and
Corollary \ref{cor:Arnal}, $\overline{\pi(u)}-\zeta$ is invertible for almost all $\pi\in \widehat{\E}$ and for all $v=\int_{\widehat{\E}}^\oplus v_\pi \, d\mu(\pi) \in \V$,
\[
R_{\omega(u)}(\zeta)v=(\overline{\omega(u)} -\zeta)^{-1}v
=\int_{\widehat{\E}}^\oplus (\overline{\pi(u)} -\zeta)^{-1}  v_\pi \, d\mu(\pi)\,.
\]
Since $\|(\overline{\pi(u)} -\zeta)^{-1} v_\pi\|_\pi \leq \|(\overline{\omega(u)} -\zeta)^{-1}\|\| v_\pi\|_\pi$ for almost all $\pi\in \widehat{\E}$ by \cite[Ch. II,\S 2, 3 (1)]{Dixmier1969-Hilbert},  the operator $(\overline{\pi(u)} -\zeta)^{-1}$ is bounded on $\V_\pi$ for almost all $\pi$.

Let $\mathcal{Z}(\mathfrak e)$ denote the center of $\U(\mathfrak e)$. In \cite[Theorem and Corollary 3]{Segal59}, Segal proved that if $u\in \mathcal{Z}(\mathfrak{e})$ then the closure $\overline{\omega(u)}$ of $\omega(u)$ is equal to the adjoint of $\omega(u^+)$. In particular, for every hermitian $u\in \mathcal{Z}(\mathfrak e)$, the operator $\omega(u)$ is essentially self-adjoint.
The spectrum of $\overline{\omega(u)}$ is therefore real.
Likewise, the spectrum of $\overline{\pi(u)}$ is real for all $\pi\in \widehat{\E}$.

The most important elements in $\mathcal{Z}(\mathfrak e)$ are the (quadratic) Casimir elements. Let $B$ be a nondegenerate symmetric bilinear form on $\mathfrak{e}$ that is invariant under the adjoint action of $\mathfrak{e}$ on itself, i.e. $B(\ad X(Y), Z) + B(Y,\ad X(Z))=0$ for all $X,Y,Z \in \mathfrak {e}$. (Usually, $B$ is the Killing form if $\mathfrak {e}$ is semisimple.)
Let $\{X_j\}$ be a basis of $\mathfrak{e}$ and let $(b^{ij})$ be the inverse of the matrix 
$(b_{ij})$ where $b_{ij}=B(X_i,X_j)$. Then $\mathcal{C}_B=\sum_{ij} b^{ij} X_j X_k$ is the Casimir element associated with $B$. It is hermitian and 
belongs to $\mathcal{Z}(\mathfrak{e})$ by the $\ad$ invariance of $B$. In fact, it belongs to 
$\mathcal{U}(\mathfrak{e})^\E$.

\begin{rems}
\begin{enumerate}
    \item The property of essential self-adjointness of  $\omega(u)$ extends to 
    other hermitian non-central elements of $\U(\mathfrak{e})$. Suppose for instance that $\E$ is a noncompact semisimple Lie group with compact center and maximal compact subgroup $\K$. Let $\U(\mathfrak{e})^\K$ denote the subspace of $\K$-invariant elements of $\U(\mathfrak{e})$. Let $\mathsf{D}$ be an $\E$-invariant dense subspace of $\V$ contained in $\V^\infty$. Then the restriction $\omega(u)|_{\mathsf{D}}$ of $\omega(u)$ to $\mathsf{D}$ is essentially self-adjoint for every hermitian $u\in \U(\mathfrak{e})^\K$. See \cite[Corollary 3]{Segal59}. Nevertheless, there are hermitian $u\in \U(\mathfrak{e})$ for which $\omega(u)$ is not essentially self-adjoint. We refer to to \cite[Section 10.2]{Schmuedgen90} for additional information and references, and to \cite{Arnal76} for some counterexamples. 
    \item 
Recall that the G{\aa}rding subspace of $\V$ is defined as the subspace of $\V^\infty$ consisting of the finite linear combinations of the vectors $\omega(f)v$ for $f\in C_c^\infty(E)$ and $v\in \V$. Here $\omega(f)=\int_\E f(g)\omega(g)\, dg$ and $dg$ is a fixed left-invariant Haar measure on $\E$.
A remarkable theorem, proven by Dixmier and Malliavin \cite{DixmierMalliavin78}, is that $\V^\infty$ coincides with the G{\aa}rding subspace of $\V$.
\end{enumerate}
\end{rems}

By Segal's infinitesimal version of Schur's lemma, if $(\pi,V_\pi)$ is unitary and irreducible and $u\in \mathcal{Z}(\mathfrak{e})$, then $\pi(u)$ acts on $V_\pi^\infty$ as a real scalar multiple of the identity. This yields the following corollary. 

\begin{cor}
Let $(\omega,\V)$ with $\V=\int_{\widehat{\E}}^\oplus \V_\pi \, d\mu(\pi)$ be as above and let $u\in \mathcal{Z}(\mathfrak{e})$ be hermitian. Then, for every $\pi\in \widehat{\E}$
there is a constant $\lambda_{\omega(u),\pi}\in \R$ such that for all $\zeta\in \C\setminus \R$ and $v=\int_{\widehat{\E}}^\oplus v_\pi \,d\mu(\pi) \in \V$, we have
\[
R_{\omega(u)}(\zeta)v=(\overline{\omega(u)} -\zeta)^{-1}v
=\int_{\widehat{\E}}^\oplus (\lambda_{\omega(u),\pi} -\zeta)^{-1}v_\pi \, d\mu(\pi)\,.
\]
\end{cor}

Considered as a bounded linear operator on $\V$, the resolvent $R_{\omega(u)}$ cannot be extended across the spectrum of $\overline{\omega(u)}$. However,
restricting $R_{\omega(u)}(\zeta)$ to a dense linear subspace $\Uv$ might allow it. 
More precisely, consider a linear topological space $\Uv$, dense in $\V$, endowed with a locally convex topology that is finer than the one induced from $\V$, and let $\Uv'$ the topological antidual space of $\Uv$, i.e. the set of continuous conjugate-linear functionals on $\Uv$. 
Endow $\Uv'$ with the  weak topology. 
We obtain the 
continuous inclusions
\begin{equation}
\label{rigged-spaces}    
\Uv\subseteq \V\subseteq \Uv'\,,
\end{equation}
where the inclusion $\V\subseteq \Uv'$ 
is the natural one, namely $v\in \V$ is identified with the 
functional $\Uv \ni w \mapsto \inner{v}{w} \in \C$ in $\Uv'$.
A double inclusion as in \eqref{rigged-spaces} is often called a rigged Hilbert space (RHS), also known as an equipped Hilbert space, or a Gelfand triplet. 

We shall also consider the linear dual of $\Uv$, denoted by $\Uv^*$, endowed with the weak topology. 

If $\X$ is a manifold endowed with a regular Borel measure, and $\V=\L^2(\X)$, then we have the antilinear isomorphism
\begin{equation}
\label{conjugation}
\L^2(\X) \ni f \mapsto \overline{f} \in \L^2(\X).
\end{equation}
Suppose that $\Uv=C_c^\infty(\X)$ is the space of compactly supported smooth functions on 
$\X$. 
Then \eqref{conjugation} composed with the inclusion $\V=\L^2(\X) \subseteq \Uv'$ gives a continuous linear embedding of $\L^2(\X)$ into $C_c^\infty(\X)^*$. We obtain then the usual construction from distribution theory
\[
C_c^\infty(\X) \subseteq \L^2(\X) \subseteq C_c^\infty(\X)^*\,.
\]
Let us suppose that we are in this case, i.e. $\omega$ is a unitary representation of $\E$ on $\L^2(\X)$. We also suppose that $\omega(u)$ has spectrum equal to $[0,+\infty)$.  We are then in a situation resembling the one of the introduction: considering the resolvent $R_{\omega(u)}$ as an operator 
\[
\C^+ \in z \longrightarrow {\rm Hom}(C_c^\infty(\X), C_c^\infty(\X)^*)\,,
\]
it might admit a holomorphic or meromorphic extension across $\R$ (possibly to a Riemann surface). 
If the extension is meromorphic, then the poles of the meromorphically extended resolvent are the 
resonances of the operator $\omega(u)$.

Notice that this might not be the most general setting one can consider: it is the one suggested by the examples presented in the introduction, in particular the case of the Laplacian on Riemannian symmetric cases of the noncompact type. Moreover, even the choice of $C_c^\infty(\X)$
is not canonical, but convenient to apply Paley-Wiener type theorems.

Suppose now that $u\in \mathcal{U}(\mathfrak{e})^\E$. Then, for every $z\in \C^+$, the resolvent $R_{\omega(u)}(z)$ intertwines the action of $\E$ via $\omega$ on $C_c^\infty(\X)$ and the extended action (also called $\omega$) on $C_c^\infty(\X)^*$. 
Assume that the resolvent of $\omega(u)$ extends meromorphically across $\R$ and that $z_0$ is a resonance. The same happens for every meromorphic continuation of the resolvent. 

The operator
\[
C_c^\infty(\X) \in v \longrightarrow \Res_{z=z_0} R_{\omega(u)} v \in C_c^\infty(\X)^*
\]
is called the residue operator at $z_0$. 
Since $\omega$ is a representation (and hence strongly continuous)
\[
\omega(g)\circ \Res_{z=z_0} R_{\omega(u)}  = \Res_{z=z_0} \left( \omega(g)\circ R_{\omega(u)} 
\right) \qquad (g\in \E)\,.
\]
The group $\E$ therefore acts on the range of the residue operator. So this range is an $\E$-module, called the residue representation. These are the objects we are studying in this article. 

\section{The pair $(\Og_{1,1},\Sp_2(\R))$}
\label{sectionO11-Sp2R}
\subsection{Action of the groups}
The group $\Og_{1,1}$ is the subgroup of $\GL_2(\R)$ generated by $\SOg_{1,1}$ and the element
\begin{equation}
\label{element s}
s=\begin{pmatrix}
0 & 1\\
1 & 0
\end{pmatrix}\,,
\end{equation}
where $\SOg_{1,1}$ is realized as the group of all matrices of the form
\[
h_a=\begin{pmatrix}
a & 0\\
0 & a^{-1}
\end{pmatrix}
\qquad (a\in\R^\times)\,.
\]
Then
\[
s^2=1\,,\qquad sh_as^{-1}=h_{a^{-1}}\qquad (a\in\R^\times)\,.
\]
The group structure of $\SOg_{1,1}$ together with this last relation determines the group structure of $\Og_{1,1}$.

Identify 
\[
\SOg_{1,1}\ni h_a \equiv a\in\R^\times\,.
\]
 The unitary dual $\widehat{\SOg_{1,1}}$ of $\SOg_{1,1}\equiv \R^\times \equiv \R_{>0}\times \Zb/2\Zb$ consists of the characters $\chi_{\varepsilon,\lambda}$ with $\varepsilon\in \{0,1\}$ and $\lambda\in\R$, where 
\[
\chi_{\varepsilon,\lambda}(h_a)=|a|^{i\lambda}\left(\frac{a}{|a|}\right)^\varepsilon \qquad (a\in\R^\times)\,. 
\]
For $\lambda>0$, set $\pi_{\varepsilon,\lambda}=\Ind_{\SOg_{1,1}}^{\Og_{1,1}} \chi_{\varepsilon,\lambda}$. This is the two-dimensional irreducible unitary representation of $\Og_{1,1}=\SOg_{1,1}\sqcup s\SOg_{1,1}$
determined by 
\begin{align*}
&\pi_{\varepsilon,\lambda}(h_a)=\left(\frac{a}{|a|}\right)^\varepsilon \begin{pmatrix}
|a|^{i\lambda} & 0 \\
0 & |a|^{-i\lambda}
\end{pmatrix} \qquad (a\in \R^\times)\,\\
&\pi_{\varepsilon,\lambda}(s)=s\,.
\end{align*}
Choosing $\varepsilon,\delta\in \{0,1\}$, one obtains four one-dimensional unitary representations of $\Og_{1,1}$ 
by setting 
\[
\pi_{0;\varepsilon,\delta}(\eta h_a)=\det(\eta)^\delta \chi_{\varepsilon,0}(h_a)
\qquad (a\in\R^\times,\; \eta\in \{1,s\}).
\]
Notice that $\pi_{0;0,1}(\eta h_a)=\det(\eta)$ is the determinant representation. These representations
exhaust the unitary dual $\widehat{\Og_{1,1}}$ of $\Og_{1,1}$.

Let $\X=M_{1,2}(\R)$ be the space of matrices consisting of one row of length two with real entries. We define an action $\omega_0$ of the group $\SOg_{1,1}$ on $\L^2(\X)$ as follows:
\[
\omega_0(h_a)v(x)=|a|^{-1}v(a^{-1}x) \qquad (a\in\R^\times\,,\ v\in \L^2(\X)\,,\ x\in\X)\,.
\]
It is easy to check that this action preserves the $\L^2$-norm.
Also, let
\begin{equation}\label{symplesticFourierTransform}
\omega_0(s)v(x')=\int_\X e^{-2\pi i x'jx^t}\,v(x)\,dx \qquad (v\in \L^2(\X)\,,\ x'\in\X)\,,
\end{equation}
where 
\begin{equation}
\label{J}
j=\begin{pmatrix}
0 & 1\\
-1 & 0
\end{pmatrix}\,.
\end{equation}
Then
\begin{equation}\label{symplecticFourierTransform}
\omega_0(s)=R(j)\fo=\fo R(j)\,,
\end{equation}
where
\begin{equation}\label{usualFourierTransform}
\fo v(x')=\int_\X e^{-2\pi i x'x^t}\,v(x)\,dx \qquad (v\in \L^2(\X)\,,\ x'\in\X)
\end{equation}
is the usual Fourier transform and 
\[
R(g)v(x)=v(xg)\qquad (g\in\GL_2(\R)\,,\ v\in \L^2(\X)\,,\ x\in\X)\,.
\]
In particular, $\omega_0(s)$ is a unitary operator.
Since $\fo^2=R(-1)$ we see that
\[
\omega_0(s)^2=R(j) \fo\fo R(j)=R(j) R(-1) R(j)=I\,.
\]
Furthermore, a straightforward computation shows that
\begin{equation}\label{omegasas-1}
\omega_0(s)\omega_0(h_a)\omega_0(s)^{-1}=\omega_0(h_{a^{-1}}) \qquad (a\in\R^\times)\,.
\end{equation}
Therefore the above formulas define a unitary representation $(\omega_0,\L^2(\X))$ of the group $\Og_{1,1}$. The group $\Sp_2(\R)=\SL_2(\R)$ acts on $\L^2(\X)$ via the right translations $R$:
\begin{equation}\label{actionofSp}
\omega_0(g)v(x)=v(xg)\qquad (g\in\Sp_2(\R)\,,\ v\in \L^2(\X)\,,\ x\in\X)\,.
\end{equation}
This action is unitary and the two actions commute. Thus $(\omega_0,\L^2(\X))$ may be viewed as a unitary representation of the group $\Og_{1,1}\times\Sp_2(\R)$, where we identify $\Og_{1,1}=\Og_{1,1}\times\{1\}$ and $\{1\}\times\Sp_2(\R)=\Sp_2(\R)$.

\subsection{The Casimir elements and the Capelli operators}
Let 
\[
h=
\begin{pmatrix}
1 & 0\\
0 & -1
\end{pmatrix}\,.
\]
Then the Lie algebra of $\SOg_{1,1}$ is $\mathfrak s\mathfrak o_{1,1}=\R h$. By taking the derivative along one parameter subgroups at the origin, see \eqref{differentiation}, we see that
\[
\omega_0(h)=-x\partial_x-y\partial_y-1\,,
\]
where we denote a typical element of $\X$ by $(x,y)$.
Moreover, let 
\begin{equation}
\label{hee}
e^+=\left(
\begin{array}{llll}
0 & 1\\
0 & 0
\end{array}
\right)
\,,\ \ \ 
e^-=\left(
\begin{array}{llll}
0 & 0\\
1 & 0
\end{array}
\right)
\,.
\end{equation}
Then $\mathfrak s\mathfrak p_2(\R)=\R h+\R e^+ +\R e^-$ and $\Ca'=h^2-2h+4e^+e^-\in \U(\mathfrak s\mathfrak p_2(\R))$ is the Casimir element. Also, one may think of $\mathcal C=h^2$ as a Casimir element in $\U(\mathfrak o_{1,1})$. A straightforward computation shows that
\begin{equation}\label{Capelli1}
\omega_0(\mathcal C)=(x\partial_x+y\partial_y+1)^2=\omega_0(\mathcal C')+1\,.
\end{equation}
This is one of Capelli's identities. (For a general story, see \cite{HoweUmeda}.) 
Set
\begin{equation}\label{PositiveCapelli1}
\mathcal C^+=-(x\partial_x+y\partial_y+1)^2
\end{equation}
Notice that 
$\mathcal{E}=x\partial_x+y\partial_y$ is the Euler operator, with formal adjoint $\mathcal{E}^*=-\mathcal{E}-2$. So $\mathcal C^+=(\mathcal{E}+1)^*(\mathcal{E}+1)$ is self-adjoint and positive. 
We would like to think of $\mathcal C^+$ as of ``the positive Capelli operator''. The Schwartz space $\mathcal{S}(\X)$ is an $\Og_{1,1}\times
\Sp_2(\R)$-invariant dense subspace of the space of smooth vectors of 
the representation $\omega_0$ of $\Og_{1,1}\times \Sp_2(\R)$.
\subsection{Direct integral decomposition of the restriction of 
$(\omega_0, \L^2(\X))$ to $\Og_{1,1}$}
\begin{lem}\label{Plancherel1}
For $\lambda\in\C$ and $v\in C_c^\infty(\X\setminus\{0\})$ define
\begin{equation}\label{Plancherel1.1}
v_\lambda(w)=\int_{\R_{>0}} a^{-1-i\lambda}v(a^{-1}w)\,d^\times a \qquad (w\in \X\setminus \{0\})\,,
\end{equation}
where $d^\times a=\frac{da}{a}$ is the Haar measure on the multiplicative group $\R_{>0}$. Then $v_\lambda$ 
is a homogeneous function of degree $-1-i\lambda$,
that is $v_\lambda(tw)=t^{-1-i\lambda}v_\lambda(w)$ for all $t>0$ and $w\in \X\setminus\{0\}$. For fixed $w$, $v_\lambda(w)$ is an entire function of Paley-Wiener type in $\lambda\in\C$. Moreover 
\begin{equation}\label{Plancherel1.2}
v(w)=\frac{1}{2\pi}\int_\R v_\lambda(w)\,d\lambda
\end{equation}
and
\begin{equation}\label{Plancherel1.3}
\int_\X u(x)\overline{v(x)}\,dx=\frac{1}{2\pi}\int_\R\int_{S^1} u_\lambda(\sigma)\overline{v_\lambda(\sigma)}\,d\sigma\,d\lambda \qquad (u, v\in C_c^\infty(\X\setminus\{0\}))\,,
\end{equation}
where $S^1\subseteq \X$ is the unit circle centered at the origin and $d\sigma$ is the rotation invariant measure on $S^1$ normalized so that the total length of $S^1$ is $2\pi$. 
\end{lem}
\begin{prf}
The first two claims are immediate by change of variables and because, for a fixed $w\in\X\setminus\{0\}$, the function $\R_{>0}\ni a\mapsto a^{-1}v(a^{-1}w)\in\C$ is smooth and compactly supported. The right-hand side of \eqref{Plancherel1.2}  is equal to
\begin{equation*}
\frac{1}{2\pi}\int_{\R_{>0}}a^{-1}\int_\R a^{-i\lambda}\,d\lambda\, v(a^{-1}w)\,d^\times  a
=\int_{\R_{>0}}a^{-1} \delta_1(a)\, v(a^{-1}w)\,d^\times  a
= v(w)\,.
\end{equation*}
The right-hand side of \eqref{Plancherel1.3}  is equal to
\begin{align*}
\frac{1}{2\pi}\int_{\R}\int_{S^1}& u_\lambda(\sigma)\overline{v_\lambda(\sigma)}\,d\sigma\,d\lambda\\
&=\frac{1}{2\pi}\int_{S^1}\int_{\R}\int_{\R_{>0}}\int_{\R_{>0}} a^{-1-i\lambda}b^{-1+i\lambda} u(a^{-1}\sigma) \overline{v(b^{-1}\sigma)}\,d^\times b\,d^\times a\,d\lambda\,d\sigma\\
&=\int_{S^1}\int_{\R_{>0}}\int_{\R_{>0}} \delta_1(ab^{-1})a^{-1}b^{-1} u(a^{-1}\sigma) \overline{v(b^{-1}\sigma)}\,d^\times b\,d^\times a\,d\sigma\\
&=\int_{S^1}\int_{\R_{>0}} a^{-2} u(a^{-1}\sigma) \overline{v(a^{-1}\sigma)}\,d^\times a\,d\sigma\\
&=\int_{S^1}\int_{\R_{>0}} a^{2} u(a\sigma) \overline{v(a\sigma)}\,d^\times a\,d\sigma\,,
\end{align*}
which coincides with the left-hand side.
\end{prf}
\begin{lem}\label{Plancherel2}
For $\lambda\in \C$ let $C_{\lambda}^\infty(\X\setminus\{0\})\subseteq C^\infty(\X\setminus\{0\})$ denote the subspace of functions homogeneous of degree $-1-i\lambda$. Then \eqref{Plancherel1.1} defines a continuous surjective map,
\begin{equation}\label{Plancherel2.1}
C_c^\infty(\X\setminus\{0\})\ni v\mapsto v_\lambda\in C_{\lambda}^\infty(\X\setminus\{0\})\,.
\end{equation}
Furthermore,
\begin{equation}\label{Plancherel2.2}
\int_\X v(w)u(w)\,dw=\frac{1}{2\pi}\int_\R\int_{S^1}v_\lambda(\sigma)u_{-\lambda}(\sigma)\,d\sigma\,d\lambda \qquad (u,v\in C_c^\infty(\X\setminus\{0\}).
\end{equation}
\end{lem}
\begin{proof}
The continuity and the surjectivity of \eqref{Plancherel2.1} follow from \cite[(3.2.21)-(3.2.23)]{Hormander}. The last equation is a straightforward consequence of \eqref{Plancherel1.3}: 
\begin{align*}
\int_\X v(w)u(w)\,dw&=\int_\X v(w)\overline{\overline{u}(w)}\,dw
=\frac{1}{2\pi}\int_\R\int_{S^1} v_\lambda(\sigma)\overline{\overline{u}_\lambda(\sigma)}\,d\sigma\,d\lambda\\
&=\frac{1}{2\pi}\int_\R\int_{S^1} v_\lambda(\sigma)u_{-\lambda}(\sigma)\,d\sigma\,d\lambda \qquad (u, v\in C_c^\infty(\X\setminus\{0\}))\,.
\end{align*}
\end{proof}
\begin{cor}\label{Plancherel3}
For $\lambda\in\C$ let $\L^2_\lambda(\X)$ denote the closure of $C_{\lambda}^\infty(\X\setminus\{0\})$ with respect to the $L^2$-norm on $S^1$:
\begin{equation}
\label{inner-lambda}
\| v_\lambda\|_\lambda=\left(\int_{S^1} |v_\lambda(\sigma)|^2\,d\sigma\right)^{1/2}\qquad (v_\lambda\in C_{\lambda}^\infty(\X\setminus\{0\}))\,.
\end{equation}
Then
\begin{equation}\label{Plancherel3.0}
\L^2(\X)=\frac{1}{2\pi}\int_\R^\oplus \L^2_\lambda(\X)\,d\lambda
\end{equation}
is the decomposition of the Hilbert space $\L^2(\X)$ into the direct integral of the Hilbert spaces $\L^2_\lambda(\X)$ with the Plancherel measure $\frac{1}{2\pi}\,d\lambda$.
\end{cor}

Notice that every element of $\L^2_\lambda(\X)$ can be written in the form 
$v_\lambda(w)=r^{-1-i\lambda}f(\sigma)$ where $w=r\sigma$ with  $(r,\sigma)\in \R_{>0}\times S^1$ and $f\in \L^2(S^1)$.

The transformation \eqref{Plancherel1.1} maps odd functions to odd functions and even functions to even functions. 

For each $\lambda\in\R$, let $\L^2_{0,\lambda}(\X)\subseteq \L^2_\lambda(\X)$ be the subspace of even functions and let $\L^2_{1,\lambda}(\X)\subseteq \L^2_\lambda(\X)$ be the subspace of odd functions. 
Then 
\[
\L^2(\X)=\frac{1}{2\pi}\int_\R^\oplus \left(\L^2_{0,\lambda}(\X)\oplus \L^2_{1,\lambda}(\X)\right)\,d\lambda\,.
\]
Each $v_\lambda\in C_{\lambda}^\infty(\X\setminus\{0\})$ is a homogeneous function of degree $-1-i\lambda$. Hence it extends uniquely to a homogeneous, and hence tempered distribution on $\X$, see \cite[Theorem 3.2.3 and 7.1.18]{Hormander}. We write $v_\lambda=v_{0,\lambda}+v_{1,\lambda}$ for the decomposition of $v_\lambda\in \L^2_\lambda(\X)=\L^2_{0,\lambda}(\X)\oplus \L^2_{1,\lambda}(\X)$. 
From now on (in this section) we view $\L^2_{\lambda}(\X)$ as a subspace of the tempered distributions $\Ss^*(\X)$,
\[
\L^2_{\lambda}(\X)\subseteq \Ss^*(\X)\,
\] 
and extend the action $\omega_0$ of $\Og_{1,1}$ 
to $\L^2_{\lambda}(\X)\subseteq\Ss^*(\X)$ by dualizing the action 
on $\Ss(\X)\subseteq \L^2(\X)$, that is
\[
(\omega_0(g)v_\lambda)(u)=v_\lambda(\omega_0(g^{-1})u) \qquad (g\in \Og_{1,1}\,, v_\lambda\in \L^2_{\lambda}(\X)\,, u\in \Ss(\X))\,.
\]
The reason is that we want to apply the Fourier transform $\omega_0(s)$, see \eqref{symplesticFourierTransform}, to elements of $\L^2_{\lambda}(\X)$.

In particular, the Fourier transform of $v_\lambda\in\L^2_\lambda(\X)$ is homogeneous of degree $-1+i\lambda$, see \cite[Theorem 7.1.24]{Hormander}. 
Hence, for $\varepsilon\in \{0,1\}$,
\[
\omega_0(s):\L^2_{\varepsilon,\lambda}(\X)\to \L^2_{\varepsilon,-\lambda}(\X)\,.
\]
The spaces $\L^2_{\varepsilon,\lambda}(\X)$ are isotypic for the action of $\SOg_{1,1}$ via $\omega_0$, 
as can be seen from the formulas
\begin{equation}
\label{omegaavlambdapm}
\omega_0(h_a) v_{0,\lambda}= |a|^{i\lambda} v_{0,\lambda}\,,\ \ \ \omega_0(h_a) v_{1,\lambda}= |a|^{i\lambda}\frac{a}{|a|} v_{1,\lambda}
\qquad (a\in\R^\times\,, \ v_{\varepsilon,\lambda}\in \L^2_{\varepsilon, \lambda}(\X))\,.
\end{equation}
Hence $\L^2_{\varepsilon, \lambda}(\X)\oplus \L^2_{\varepsilon,-\lambda}(\X)$ is preserved under the action of $\Og_{1,1}$. If $\lambda>0$, then this representation is isotypic, direct integral of a single 2-dimensional irreducible representation, which we denote by $(\pi_{\varepsilon,\lambda}, \V_{\varepsilon,\lambda})$. Indeed, fix $v_{\varepsilon,\lambda}\in \L^2(\X)_{\varepsilon,\lambda}$ and set $\V_{\varepsilon,\lambda}=\C v_{\varepsilon,\lambda} \oplus \C v_{\varepsilon,-\lambda}$, where $v_{\varepsilon,-\lambda}=\omega_0(s)v_{\varepsilon, \lambda}$.
Let $\mathcal{B}_{\varepsilon,\lambda}=\{v_{\varepsilon,\lambda},v_{\varepsilon,-\lambda}\}$. Then the matrix of $\omega_0(s)|_{\V_{\varepsilon,\lambda}}$ with respect to $\mathcal{B}_{\varepsilon,\lambda}$ 
is $s$. For $a\in \R^\times$, by \eqref{omegasas-1} and \eqref{omegaavlambdapm}, 
\[
\omega_0(h_a)|_{\V_{0,\lambda}}=\begin{pmatrix}
|a|^{i\lambda} & 0 \\ 0 & |a|^{-i\lambda}
\end{pmatrix} \qquad \textrm{with respect to $\mathcal{B}_{0,\lambda}$}
\]
\[
\omega_0(h_a)|_{\V_{1,\lambda}}=\frac{a}{|a|}
\begin{pmatrix}
|a|^{i\lambda} & 0 \\ 0 & |a|^{-i\lambda}
\end{pmatrix} \qquad \textrm{with respect to $\mathcal{B}_{1,\lambda}$}\,.
\]
Thus 
$\omega_0|_{\V_{0,\lambda}}=\pi_{0,\lambda}$ and 
$\omega_0|_{\V_{1,\lambda}}=\pi_{1,\lambda}$.
We have therefore proved the following corollary. 

\begin{cor}\label{O11decomosition}
For $\varepsilon\in \{0,1\}$ and $\lambda>0$, 
\[
\L^2(\X)_{\pi_{\varepsilon,\lambda}}=\L^2(\X)_{\varepsilon,\lambda} \oplus\L^2(\X)_{\varepsilon,-\lambda}
\]
is an isotypic representation of $\Og_{1,1}$ of type $\pi_{\varepsilon,\lambda}$.

The restriction of the representation $(\omega_0, \L^2(\X))$ to $\Og_{1,1}$ decomposes into direct integral of irreducible unitary representations as follows,
\[
\L^2(\X)=\int_{\widehat{\Og}_{1,1}} \L^2(\X)_\pi\,d\mu(\pi)\,,
\]
where $d\mu(\pi_{\varepsilon,\lambda})=
\frac{d\lambda}{2\pi}$ for $(\varepsilon,\lambda)\in \{0,1\}\times \R_{>0}$, $\mu(\pi_{\varepsilon,0,\delta})=0$ 
for $(\varepsilon,\delta)\in \{0,1\}^2$
and
$\L^2(\X)_\pi$ denotes the isotypic component of type $\pi$.
\end{cor}
\subsection{The resonance}
\label{subsection:resonanceO11SL}
\begin{lem}\label{resonance1}
Recall the densely defined differential operator $\mathcal C^+$ on $\L^2(\X)$, see \eqref{PositiveCapelli1}. For $z\in\C$ with $\Im z>0$ the operator $\mathcal C^+-z^2$ is invertible  with  inverse 
\begin{equation}\label{resonance11}
(\mathcal C^+-z^2)^{-1}:\L^2(\X) \to \L^2(\X)
\end{equation}
given, in terms of \eqref{Plancherel3.0}, by
\[
(\mathcal C^+-z^2)^{-1}\left(\frac{1}{2\pi}\int_\R v_\lambda\,d\lambda\right)=\frac{1}{2\pi}\int_\R (\lambda^2-z^2)^{-1}v_\lambda\,d\lambda\,.
\]
\end{lem}
\begin{proof}
This follows from the straightforward fact that $\Ca^+ v_\lambda=\lambda^2 v_\lambda$.
\end{proof}
\begin{pro}\label{resonance00}
If we shrink the domain and expand the range of the map \eqref{resonance11}
\begin{equation}\label{resonance13}
(\mathcal C^+-z^2)^{-1}:C_c^\infty(\X\setminus\{0\}) \to C_c^\infty(\X\setminus\{0\})^*
\end{equation}
by the formula
\begin{equation}\label{resonance14}
((\mathcal C^+-z^2)^{-1} v)(u)=\int_\X ((\mathcal C^+-z^2)^{-1} v)(w) u(w)\,dw \qquad (u,v\in C_c^\infty(\X\setminus\{0\}))\,,
\end{equation}
then \eqref{resonance13} extends from $\Im z>0$ to a meromorphic function of $z\in\C$ with a single simple pole at $z=0$, with residue operator given by
\[
{\Res}_{z=0}((\mathcal C^+-z^2)^{-1} v)
=\frac{i}{2}v_0\,.
\]
Here $v_0$ is viewed as a distribution on $\X\setminus\{0\}$ via integration against $dw$. This distribution extends uniquely to a homogeneous distribution on $\X$.
\end{pro}
\begin{prf}
The equality \eqref{Plancherel2.2} shows that \eqref{resonance14} means that
\begin{equation}\label{resonance15}
((\mathcal C^+-z^2)^{-1} v)(u)=\frac{1}{2\pi}\int_\R\int_{S^1} (\lambda^2-z^2)^{-1} v_\lambda(\sigma) u_{-\lambda}(\sigma)\,d\sigma\,d\lambda\,.
\end{equation}
Notice that
\[
(\lambda^2-z^2)^{-1}=-\frac{1}{2z}\left(\frac{1}{z-\lambda}+\frac{1}{z+\lambda}\right)\,.
\]
Hence, the right hand side of \eqref{resonance15} is equal to
\begin{equation}\label{resonance16}
-\frac{1}{4\pi z}\left(\int_\R\int_{S^1} \frac{1}{z-\lambda} v_\lambda(\sigma) u_{-\lambda}(\sigma)\,d\sigma\,d\lambda
+\int_\R\int_{S^1} \frac{1}{z+\lambda} v_\lambda(\sigma) u_{-\lambda}(\sigma)\,d\sigma\,d\lambda\right)\,.
\end{equation}
The function in the parenthesis extends to an entire function of $z$. Indeed, since the function 
\[
\C\ni \lambda\mapsto \int_{S^1}  v_\lambda(\sigma) u_{-\lambda}(\sigma)\,d\sigma\in\C
\]
is of Paley-Wiener type, we may pick any $N>0$ and, using Cauchy's theorem, show that 
\begin{align}\label{resonance17}
\int_\R\int_{S^1}& \frac{1}{z-\lambda} v_\lambda(\sigma) u_{-\lambda}(\sigma)\,d\sigma\,d\lambda
+\int_\R\int_{S^1} \frac{1}{z+\lambda} v_\lambda(\sigma) u_{-\lambda}(\sigma)\,d\sigma\,d\lambda\notag\\
&=\int_{\R-iN}\int_{S^1} \frac{1}{z-\lambda} v_\lambda(\sigma) u_{-\lambda}(\sigma)\,d\sigma\,d\lambda
+\int_{\R+iN}\int_{S^1} \frac{1}{z+\lambda} v_\lambda(\sigma) u_{-\lambda}(\sigma)\,d\sigma\,d\lambda\,.
\end{align}
The right hand side of \eqref{resonance17} is a holomorphic function for $\Im z>-N$. Therefore \eqref{resonance16} is a meromorphic function with a unique simple pole at zero. The residue at zero is equal to
\begin{align*}
-\frac{1}{4\pi}&\left(\int_{\R-iN}\int_{S^1} \frac{1}{-\lambda} v_\lambda(\sigma) u_{-\lambda}(\sigma)\,d\sigma\,d\lambda
+\int_{\R+iN}\int_{S^1} \frac{1}{\lambda} v_\lambda(\sigma) u_{-\lambda}(\sigma)\,d\sigma\,d\lambda\right)\notag\\
&=\frac{1}{4\pi}\int_{S^1}\left(\int_{\R-iN} \frac{1}{\lambda} v_\lambda(\sigma) u_{-\lambda}(\sigma)\,d\lambda
-\int_{\R+iN}\frac{1}{\lambda} v_\lambda(\sigma) u_{-\lambda}(\sigma)\,d\lambda\right)\,d\sigma\notag\\
&=\frac{1}{4\pi}\int_{S^1}\int_{|\lambda|=N} \frac{1}{\lambda} v_\lambda(\sigma) u_{-\lambda}(\sigma)\,d\lambda\,d\sigma
=\frac{i}{2}\int_{S^1}v_0(\sigma) u_0(\sigma)\,d\sigma\,, 
\end{align*}
where we used again Cauchy's theorem and the Paley-Wiener property of $v_\lambda$ and $u_{-\lambda}$ (see Lemma \ref{Plancherel1}), and finally Cauchy's integral formula. 
Thus
\[
\Res_{z=0}((\mathcal C^+-z^2)^{-1} v)(u)=\frac{i}{2}\int_{S^1}v_0(\sigma) u_0(\sigma)\,d\sigma
=\frac{i}{2}\int_{\X}v_0(w) u(w)\,dw\,.
\]
\end{prf}

\subsection{The resonance representation}

By Proposition \ref{resonance00}, the resonance space at $\lambda=0$ is 
\[
\{v_0 \in C^\infty(\X\setminus\{0\}): v \in C_c^\infty(\X\setminus\{0\})\}\,.
\]
Its completion with respect to the inner product \eqref{inner-lambda} is the Hilbert space $\L^2_0(\X)$.
In this subsection we take a look at this space as representation of $\Og_{1,1}$.

The elements of $\L^2_0(\X)$ are of the form 
$r^{-1}f(e^{i\theta})$ where $w=re^{i\theta}$ with $(r,e^{i\theta})\in \R_{>0}\times S^1$ and $f\in\L^2(S^1)$. By the $\L^2$-Fourier expansion 
$f(e^{i\theta})\sim \sum_{k\in \Zb} \Hat{f}(k)e^{ik\theta}$, it suffices to consider the action of $\omega_0$ on $r^{-1}e^{ik\theta}$, $k\in \Zb$.

\begin{lem}\label{rsonance 23}
The following formulas hold:
\begin{equation}
\label{omegas1}
\omega_0(s): r^{-1}e^{ik\theta} \mapsto r^{-1}e^{ik\theta}, \qquad \text{if $k\in \Zb, k\geq 0$}\,,
\end{equation}
\begin{equation}
\label{omegas2}
\omega_0(s): r^{-1}e^{ik\theta} \mapsto (-1)^kr^{-1}e^{ik\theta}, \qquad \text{if $k\in \Zb, k< 0$}\,.
\end{equation}
\end{lem}
\begin{proof}
For $t>0$ define $g_{k,t}(w)=f_t(r) e^{ik\theta}$, where $w=re^{i\theta}$ and $f_t(r)=r^{-1}e^{-2\pi tr}$.
Then $g_{k,t}\in \L^1(\X)$ and $\lim_{t\to 0^+} g_{k,t}=g_{k}$, where $g_{k}(re^{i\theta})=r^{-1}e^{ik\theta}$ and 
the limit is in $\Ss^*(\X)$.

As is well known, the two-dimensional Euclidean Fourier transform $\fo$, see  \eqref{usualFourierTransform}, may be expressed in terms of Bessel functions by passing to polar coordinates. In particular, by \cite[Ch. 4, Theorem 1.6]{SteinWeiss}, if $g(w)=f(r)e^{ik\theta} \in \L^1(\X)$, then 
$(\fo g)(w)=F(r)e^{ik\theta}$ where $w=re^{i\theta}$ and 
\begin{equation}
    \label{F0}
    F(r)=2\pi i^k\int_0^\infty f(\rho) J_{-k}(2\pi r\rho) \rho\,d\rho
        =2\pi (-1)^k i^k \int_0^\infty f(\rho) J_{k}(2\pi r\rho) \rho\,d\rho\,.
\end{equation}
In \eqref{F0}, $J_k$ denotes the $k$-th Bessel function of the first kind, defined for 
$k\in \Zb$ by
\[
J_k(x)=\frac{1}{2\pi} \int_0^\infty e^{ix\sin \theta} e^{-ik\theta}d\theta
\]
and satisfying $J_{-k}(x)=(-1)^k J_k(x)$.

Recall from \eqref{symplecticFourierTransform} that $\omega_0(s)=\fo R(J)$. If $w=(r\cos\theta,r\sin\theta)\equiv re^{i\theta}$, then 
$wJ=(-r\sin\theta,r\cos\theta)\equiv r e^{i(\theta+\frac{\pi}{2})}$.
Hence 
\[
(\omega_0(s)g_{k,t})(w)=\fo g_{k,t}\left(r e^{i(\theta+\frac{\pi}{2})}\right)=F_{k,t}(r) e^{i(\theta+\frac{\pi}{2})}\,,
\]
where 
\begin{equation}
    \label{Ft}
    F_{k,t}(r)=2\pi i^k\int_0^\infty e^{-2\pi t\rho} J_{-k}(2\pi r\rho)\,d\rho
        =2\pi (-1)^k i^k \int_0^\infty e^{-2\pi t\rho} J_k(2\pi r\rho)\,d\rho\,.
\end{equation}
As $t>0$, for $k>-1$
\[
2\pi \int_0^\infty e^{-2\pi t\rho} J_k(2\pi r\rho)\,d\rho
= \int_0^\infty e^{-t\rho} J_k(r\rho)\,d\rho
=\frac{(\sqrt{t^2+r^2}-t)^k}{r^k\sqrt{t^2+r^2}}
\]
by \cite[formula (8) on page 386; this formula is due to Lipschitz (1859) for $k=0$ and to Hankel (1875) for $k=\nu$ with $\Re\nu>-1$]{WatsonBessel}. 
Hence for $k\in \Zb$, $k\geq 0$, 
\begin{align*}
(\omega_0(s)g_k)(re^{i\theta})&=\lim_{t\to 0^+} (\omega_0(s)g_{k,t})(re^{i\theta})\\
 &=\lim_{t\to 0^+} F_{k,t}(r)e^{ik(\theta+\frac{\pi}{2})}\\
 &=(-1)^k i^k i^k \lim_{t\to 0^+} \frac{(\sqrt{t^2+r^2}-t)^k}{r^k\sqrt{t^2+r^2}}\\
 &= r^{-1}e^{ik\theta}\,,
\end{align*}
which is \eqref{omegas1}. If $k<0$, then one applies the above to the first formula in \eqref{Ft} and gets
\[
(\omega_0(s)g_k)(re^{i\theta})=(-1)^k r^{-1}e^{ik\theta}\,,
\]
which is \eqref{omegas2}.
\end{proof}

\begin{lem}\label{resonance 23bis}
The following formulas hold:
\[
\omega_0(h_a): r^{-1}e^{ik\theta} \mapsto r^{-1}e^{ik\theta}, \qquad \text{if $k\in \Zb$, $k$ even}\,,
\]
\[
\omega_0(h_a): r^{-1}e^{ik\theta} \mapsto  \frac{a}{|a|} r^{-1}e^{ik\theta}, \qquad \text{if $k\in \Zb$, $k$ odd}\,.
\]
\end{lem}
\begin{proof}
If $g_k$ is defined by $g_k(re^{i\theta})=r^{-1}e^{ik\theta}$, then 
\[
|a|^{-1}g_k(a^{-1}re^{i\theta})=
\begin{cases}
r^{-1}e^{ik\theta} &\text{if $a>0$}\\
(-1)^kr^{-1}e^{ik\theta} &\text{if $a<0$}\,.
\end{cases}
\]
\end{proof}

\begin{cor}\label{resonance 24}
The restriction of $\omega_0$ to $\L^2_0(\X)$ decomposes as the direct sum
\[
\L^2_0(\X)=\L^2_{0;0,0}(\X)\oplus\L^2_{0;1,0}(\X)\oplus\L^2_{0;1,1}(\X)
\]
of isotypic $\Og_{1,1}$-representations. Explicitly,
\begin{align*}
\L^2(\X)_{0;0,0}&=\bigoplus_{\substack{k\in \Zb\\ \textup{$k$ even}}}
\C r^{-1}e^{ik\theta}\,,\\
\L^2(\X)_{0;1,0}&=\bigoplus_{\substack{k\geq 0\\ \textup{$k$ odd}}}\C r^{-1}e^{ik\theta}\,,\\
\L^2(\X)_{0;1,1}&=\bigoplus_{\substack{k<0\\ \textup{$k$ odd}}}\C r^{-1}e^{ik\theta}\,.
\end{align*}
For $(\varepsilon,\delta)\in \{(0,0),(1,0),(1,1)\}$, the representation on $\L^2_{0;\varepsilon,\delta}(\X)$ is isotypic, with $1$-dimen\-sional type $\pi_{0;\varepsilon,\delta}$.  
In particular, the determinant representation $\pi_{0;0,1}$ of $\Og_{1,1}$ does not occur in the decomposition.
\end{cor}

Because of Corollary \ref{resonance 24}, $\L^2(\X)_{0;\varepsilon,\delta}$ is the $\Og_{1,1}$-isotypic component of type $\pi_{0;\varepsilon,\delta}$. Hence we write
\[
\L^2(\X)_{\pi_{0;\varepsilon,\delta}}=\L^2(\X)_{0;\varepsilon,\delta}\,.
\]
We summarize our result in the following theorem.

\begin{thm}
The resonance representation $\L^2(\X)_0$ of $\Og_{1,1}$ splits as follows:
\begin{equation}
\label{resonance21}
\L^2(\X)_0=\L^2(\X)_{\pi_{0;0,0}}\oplus \L^2(\X)_{\pi_{0;1,1}}\oplus \L^2(\X)_{\pi_{0;1,0}}\,,
\end{equation}
where $\L^2(\X)_{\pi_{0;\varepsilon,\delta}}$ is isotypic, with one dimensional type $\pi_{0;\varepsilon,\delta}$. In particular: 
\begin{enumerate}
\item
$\Og_{1,1}$ acts trivially on $\L^2(\X)_{\pi_{0;0,0}}$, 
\item
the group $\SO_{1,1}$ acts by the sign representation on $\L^2(\X)_{\pi_{0;1,0}}\oplus \L^2(\X)_{\pi_{0;1,1}}$, 
\item
the element $s\in \Og_{1,1}$ acts trivially on $\L^2(\X)_{\pi_{0;1,0}}$,
\item
the element $s\in \Og_{1,1}$ acts via multiplication by $-1$ on $\L^2(X)_{\pi_{0;1,1}}$.
\end{enumerate}
The determinant representation of $\Og_{1,1}$ does not occur in the decomposition. 
\end{thm}

Each of the spaces \eqref{resonance21} is contained in $\Ss^*(\X)$ and is preserved by the action of $\Sp_2(\R)$ via $\omega_0$, see \eqref{actionofSp}. 

The three representations on the right-hand side of 
\eqref{resonance21} are unitary representations of $\Sp_2(\R)$ and we know that the Casimir $\omega_0(\mathcal{C}')$ acts by $-1$ because of \eqref{Capelli1}. We also know their $\K$-types by
Corollary \ref{resonance 24}. If we knew they are irreducible $\Sp_2(\R)$-modules, then we would identify them by classification. Fortunately, this is a simple consequence of Howe's duality theory. 
In order to use this theory, we have to move to the metaplectic group $\wt{\Sp}_4(\R)$ and relate its Weil representation $\omega$ to $\omega_0$. This is explained in Appendix \ref{appendix:Weil representation} in general, and specifically in Appendix \ref{appendix:O11Sp2} for this pair, see \eqref{twistedrep}. In fact, $\omega_0$ agrees with $\omega|_{\wt{\Og_{1,1}}\wt{\Sp_2(\R)}}$ twisted by a character, which is a representation of $\Og_{1,1}\Sp_2(\R)$. The equality of these two representations is true by \eqref{omegaonM-det1} for $\SO_{1,1}\Sp_2(\R)$. On the other hand, if $\wt s$ denotes a preimage of $s$ under the metapletic cover, then $\omega(\wt s)$ is computed in \eqref{omegawts} and the twisting removes the $\pm$ ambiguity. Twisting by a character does not change the irreducibility. So the three representations are irreducible. By the classification of the irreducible unitary 
$\Sp_2(\R)=\SL_2(\R)$-modules, (see e.g. \cite[VI, \S 6]{LangSL2R}), one obtains the following corollary.

\begin{cor}\label{resonance 25}
The spaces \eqref{resonance21} are irreducible unitary $\omega_0(\Sp_2(\R))$-modules. Specifically, 
\begin{enumerate}
\item $\L^2(\X)_{\pi_{0;0,0}}=\L^2(\X)_{\pi_{0,1}}$ is the spherical unitary principal series $\pi_{0,1}$ on which the Casimir element
$\Ca{}'$ acts by $-1$
\item $\L^2(\X)_{\pi_{0;1,0}}=\L^2(\X)_{\Dsf^0_+}$ is the holomorphic limit of discrete series $\Dsf^0_+$\,,
\item $\L^2(\X)_{\pi_{0;1,1}}=\L^2(\X)_{\Dsf^0_-}$ is the  anti-holomorphic limit of discrete  series $\Dsf^0_-$\,.
\end{enumerate}
Hence the entire resonance space \eqref{resonance21} is not irreducible under the joint action of $\Og_{1,1}\times \Sp_2(\R)$.
It is the direct sum of three irreducible subspaces:
\[
\left(\pi_{0;0,0}\otimes \pi_{0,1}\right) \oplus
\left(\pi_{0;1,0}\otimes \Dsf^0_+\right) \oplus
\left(\pi_{0;1,1}\otimes \Dsf^0_-\right)\,.
\]
\end{cor}

\begin{rem}
As observed in the introduction, one of the motivating examples for the study of resonances is the Casimir element acting by the left-regular representation on a Riemannian symmetric space of the noncompact type. One could consider other classes of homogenous spaces.
For instance, if $\G'=\SL_2(\R)$ and $\N'=\left\{\begin{pmatrix}
1 & x \\ 0 & 1
\end{pmatrix}; x \in \R\right\}$, then the homogeneous space $\G'/\N'$ can be realized as $\X\setminus\{0\}$, where $\X=M_{1,2}(\R)$. Our computation in this section can also be interpreted in this sense.
\end{rem}

\begin{rem}
The case $(\G,\G')=(\Og_{1,1},\Sp_{2n}(\R))$ with $n>1$ could be treated in a similar way, but the result on the resonance representations would be less explicit. 
\end{rem}

\section{Decomposing the restriction to $\G$ of the (twisted) Weil representation using Harish-Chandra's Plancherel formula}
\label{section:Plancherel}
In this section we outline the method of decomposing the restriction to $\G$ of the Weil representation (twisted by a suitable character) using Harish-Chandra's Plancherel formula. This method applies to orthosymplectic dual pairs of the form 
$(\G,\G')=(\Sp_{2n}(\R), \Og_{p,p})$ or  $(\Og_{p,p},\Sp_{2n}(\R))$ in the stable range, with $\G$ the smaller member. We first recall some facts concerning these pairs. 

Any such pair $(\G,\G')$ is an irreducible real reductive dual pair of type I (see \cite{HoweRemarks}) and can be constructed
as follows. There exist real vector spaces $\V$ and $\V'$ with non-degenerate bilinear forms
$(\cdot ,\cdot)$ and $(\cdot ,\cdot)'$, one symmetric and the other skew-symmetric (or vice versa), such that $\G \subseteq \GL(\V)$ and $\G'\subseteq \GL(\V')$ are the isometry groups of $(\cdot ,\cdot )$ and $(\cdot ,\cdot )'$, respectively. Let $\Wv = \Hom_{\Bbb D}(\V',\V)$. Define a map
\[
\Hom_{\Bbb D}(\V',\V)\ni w\mapsto w^*\in \Hom_{\Bbb D}(\V,\V')
\]
by 
\[
(wv',v)=(v',w^*v)' \qquad (v\in \V,\, v'\in \V')\,.
\]
Then the formula
\[
\langle w,w'\rangle=\tr_{\Bbb D/\R}(w w'{}^*) \qquad (w,w'\in\Wv)
\]
defines a non-degenerate symplectic form on $\Wv$. We denote by 
$\Sp(\Wv)$ the symplectic group of $(\Wv, \inner{\cdot}{\cdot})$. 
The groups $\G$ and $\G'$ act on $\Wv$ by 
\begin{equation}
\label{GG'actions}
g(w)=gw \qquad \text{and} \qquad g'(w)=w {g'}^{-1} \qquad (g\in \G,\, g'\in \G',\, w\in \Wv).
\end{equation}
These actions embed $\G$ and $\G'$ as a subgroups of $\Sp(\Wv)$. 

Let $\V'=\X'\oplus \Y'$ be a complete polarization of $\V'$. 
Assuming that $(\G,\G')$ is in the stable range, with $\G$ the smaller set, means that 
$\dim_\DD \X'\geq \dim_\DD \V$. Hence, $(\Sp_{2n}(\R), \Og_{p,p})$ is in the stable range, with 
$\Sp_{2n}(\R)$ the smaller member, if and only if $p\geq 2n$. Similarly, $(\Og_{p,p}, \Sp_{2n}(\R))$ is in the stable range, with $\Og_{p,p}$ the smaller member, if and only if $n\geq 2p$. (In particular, neither $(\Og_{1,1}, \Sp_{2}(\R))$ nor $(\Sp_{2}(\R), \Og_{1,1})$ are in the stable range.)
Set 
\begin{equation}
\label{XY}
\X = \Hom_{\Bbb D}(\X', \V)\,,\qquad \Y = \Hom_{\Bbb D}(\Y', \V)\,.
\end{equation}
Then $\X$, $\Y$ are isotropic and $\Wv=\X\oplus\Y$. 

Let $\wt{\Sp}(\Wv)$ be the metaplectic group and let $\wt{\Sp}(\Wv) \in \wt{g} \mapsto g\in \Sp(\Wv)$ be the metaplectic cover, which is a double cover of $\Sp(\Wv)$. For a subgroup $\H$ of $\Sp(\Wv)$ we denote by 
$\wt{\H}$ its preimage in $\wt{\Sp}(\Wv)$.

Let $(\omega, \L^2(\X))$ be the Schr\"odinger model of the Weil representation of $\wt{\Sp}(\Wv)$ attached to the character $\chi(r)=e^{2\pi i r}$ of $\R$. See Appendix \ref{appendix:Weil representation}. The space of smooth vectors of $\omega$ is $\mathcal{S}(\X)$.
By \eqref{GG'actions} and \eqref{XY}, $\G$ preserves both $\X$ and $\Y$. 
Hence, by \eqref{omegaonM} there is a continuous group homomorphism $\det_\X^{-1/2}: \wt{\G} \to \C^\times$, with the property that $\left(\det_\X^{-1/2}(\wt{g})\right)^2=\det(g|_\X)^{-1}$ for all $\wt{g}\in \wt{\G}$, such that 
\[
\omega(\wt{g})v(x)=\det_\X^{-1/2}(\wt{g}) v(g^{-1}x) \qquad (\wt{g}\in \wt{\G}, \, v\in \mathcal{S}(\X),\, x\in \X)\,.
\]

Let $\ZZ_2$ denote the kernel of the metaplectic cover. A representation $\Pi$ of $\wt{\G}$ is called genuine if its restriction to $\ZZ_2$ is a multiple of the unique non-trivial character $\varepsilon$ of $\ZZ_2$. Only genuine representations of $\wt{\G}$ can occur in Howe's duality. 

For an irreducible unitary representation $\Pi$ of $\wt{\G}$, we denote by $\Theta_\Pi$ its distribution character. By Harish-Chandra's regularity theorem \cite[8.4.1]{WallachI}, the distribution $\Theta_\Pi$ coincides with the Haar measure on $\wt{\G}$ multiplied by a locally integrable functions (which is real analytic on the set of regular semisimple elements of $\wt{\G}$ and zero elsewhere). We identify $\Theta_\Pi$ with this function. Furthermore, we denote by $\Pi^c$ the contragredient representation of $\Pi$. Notice that, if $\Pi$ is a genuine representation of $\wt{\G}$, then the map 
\[
\wt{\G}\ni\wt{g} \mapsto \Theta_{\Pi^c}(\wt{g})\omega(\wt{g}) \in \mathcal{B}(\L^2(\X)),
\] 
where $\mathcal{B}(\L^2(\X))$ is the space of bounded linear operators on $\L^2(\X)$, is constant on the fibers of the metaplectic cover 
$\wt{\G} \to \G$ and hence defines a function on $\G$.

The following theorem was proved in \cite[Theorem 3.1]{PrzebindaUnitary}
in a more general context.

\begin{thm}\label{1993_thm3.1}
Let $\Pi$ be a genuine irreducible tempered unitary representation of $\wt\G$.
Then the formula
\begin{equation}\label{1993_thm3.1_1}
(\omega(\Theta_{\Pi^c})u,v)=\int_{\G} \Theta_{\Pi^c}(\t g)(\omega(\t g)u,v)\,dg \qquad (u,v\in \mathcal{S}(\X))
\end{equation}
defines a non-trivial, hermitian, positive semidefinite $\wt\G\cdot\wt\G'$-invariant form on 
$\mathcal{S}(\X)$. 

Let $\mathcal R$ denote the radical of this form. 
Then the $\wt\G\cdot \wt\G’$-module $\mathcal{S}(\X)/\mathcal R$, equipped with the form induced by the form \eqref{1993_thm3.1_1}, completes to
an irreducible unitary representation of $\wt\G\cdot \wt\G’$ on a Hilbert space $\mathcal H_{\omega, \Pi\otimes\Pi'}$, infinitesimally equivalent to $\Pi\otimes\Pi'$
for some $\Pi'$ in the unitary dual of $\wt \G’$. Moreover $\Pi$ corresponds to $\Pi'$ via Howe’s
correspondence.
\end{thm}
\begin{proof}
We only need to check that conditions (a) and (b) of \cite[Theorem 3.1]{PrzebindaUnitary} are satisfied. Condition (b) holds by \cite[Lemmas 3.2 and 8.6]{PrzebindaUnitary}. According to \cite[Proposition 4.11]{PrzebindaUnitary}, condition (a) -- which guarantees the absolute convergence of the integral on the right-hand side of \eqref{1993_thm3.1_1}-- is satisfied when $\Theta_\Pi$ has rate of growth $\gamma <\gamma_{\max}=\lambda_{\max}-1$ where, for a dual pair of type I, 
\[
\lambda_{\max}=\frac{\dim_\DD \V'}{r-1}\qquad \text{and}\qquad r=\frac{2\dim_\R \g}{\dim_\R \V}\,.
\]
We are supposing that $\Pi$ is tempered, which is equivalent to $\gamma=0$ \cite[5.1.1]{WallachI}.
The following table shows that the condition $\lambda_{\max}>1$ 
is always satisfied under the stable range assumption.

\begin{center}
\setlength{\extrarowheight}{2mm}
\begin{tabular}{|c|c|c|c|c|c|}
\hline
$\G$ & $\dim_\R \g$ & $\dim_\R \V$ & $r-1$ & stable range condition & $\lambda_{\max}$ 
\\[.2em]
\hline
$\Sp_{2n}(\R)$ & $n(2n+1)$ & $2n$ & $2n$ & $p\geq 2n$ & $\frac{2p-1}{2n}  %\geq 2-\frac{1}{2n}>1 
$ \\[.2em]
\hline
$\Og_{p,p}$ &  $p(2p-1)$ & $2p$ & $2p-1$ & $n\geq 2p$ & $\frac{2n}{2p-1}$\\[.4em]
\hline
\end{tabular}
\end{center}
\end{proof}

For the dual pairs $(\G,\G')$ we consider, we do not need to work with double covers. In fact, there is a unitary character $\chi_+$ of $\wt\G\wt\G'$ -- see \eqref{character-chi-plus} in Appendix \ref{appendix:Weil representation} -- such that 
\[
\omega_0=\chi_+^{-1}\omega
\]
is constant on the fibers of the metaplectic covering  $\wt\G\wt\G' \to \G\G'$ and hence defines a representation of $\G\G'$, which we denote by the same symbol $\omega_0$. 
Given representations
$\Pi$ and $\Pi'$ of $\wt\G$ and $\wt\G'$ in Howe's correspondence, then $\pi=\chi_+^{-1}\Pi$
and $\pi'=\chi_+^{-1}\Pi'$ are representations of $\G$ and $\G'$, respectively.
This gives a bijection between representations that are quotients of $\omega|_{\wt\G\wt\G'}$ and 
representations that are quotients of $\omega_0$. We refer to Appendix \ref{appendix:Weil representation} for explanations. An adapted modification yields the following corollary. 

\begin{cor}\label{omega0-1993_thm3.1}
Let $\pi$ be an irreducible tempered unitary representation of $\G$.
Then the formula
\begin{equation}\label{omega0-1993_thm3.1_1}
(\omega_0(\Theta_{\pi^c})u,v)=\int_{\G} \Theta_{\pi^c}(g)(\omega_0(g)u,v)\,dg \qquad (u,v\in \mathcal{S}(\X))
\end{equation}
defines a non-trivial, hermitian, positive semidefinite $\G\G'$-invariant form on 
$\mathcal{S}(\X)$. 

Let $\mathcal R$ denote the radical of this form. 
Then the $\G\G’$-module $\mathcal{S}(\X)/\mathcal R$, equipped with the form induced by the form \eqref{omega0-1993_thm3.1_1}, completes to
an irreducible unitary representation of $\G\G’$ on a Hilbert space $\L^2(\X)_{\pi\otimes\pi'}$, infinitesimally equivalent to $\pi\otimes\pi'$
for some $\pi'$ in the unitary dual of $\G’$. Moreover $\Pi=\chi_+ \pi$ corresponds to 
$\Pi'=\chi_+\pi'$ via Howe’s correspondence.
\end{cor}

Let $\X^{\max}$ denote the dense and open subset of $\X=\Hom_\DD(\X',\V)$ of endomorphisms of maximal rank. Since $\dim_\DD \X'\geq \dim_\DD \V$, the set $\X^{\max}$ consists of the $\mathbb{D}$-linear surjective maps $x:\X'\to \V$. Each $x\in \X^{\max}$ defines an embedding of $\G$ into $\X^{\max}$ by $g \mapsto g^{-1}x$. 

The following lemma will allow us to decompose the restriction of $\omega_0$ to $\G$ using  Harish-Chandra's Plancherel formula on $\G$.

\begin{lem}
\label{lemma:compact-support}
Let $x\in \X^{\max}$ and let $v\in C_c^\infty(\X^{\max})$. 
\begin{enumerate}
\item The $\G$-orbit $\G x$ is a closed subset of $\X$ contained in $\X^{\max}$.
\item Let $v_x:\G\to \C$ be defined by $v_x(g)=v(g^{-1}x)$. Then $v_x\in C_c^\infty(\G)$.
\end{enumerate}
\end{lem}
\begin{proof}
Identify $\X$ with the space $M_{d,m}(\DD)$ of $d\times m$ matrices with coefficients in $\DD$, where $d=\dim_\DD \V$ and $m=\dim_\DD \X'$. Then there are $a\in \GL_d(\DD)$ and $b\in\GL_m(\DD)$ such that $x=aeb$ where $e=(I_d\, | \, 0)$ and $I_d$ is the $d\times d$ identity matrix. Hence $\G x=a\G^a eb$ where $\G^a=a^{-1}\G a$. Since left multiplication by $a$ and right multiplication by $b$ are homeomorphisms of $M_{m,d}(\DD)$, then $\G x$ is closed in $M_{m,d}(\DD)$ if and only if so is $\G^a e$. The right multiplication by $e$ embeds $\End_\DD(\V)=M_d(\DD)$ into $M_{d,m}(\DD)=\X$. Then $\G^ae$ is closed because homeomorphic image of $\G^a$, which is closed as $\G$ is the isotropy subgroup of $(\cdot,\cdot)$ in $\V$.
This proves (1). For (2), we only need to comment on the support. Notice that the map $g \mapsto g^{-1}x$ is a homeomorphism of $\G$ onto the orbit $\G x$. It maps the support $\supp v_x$ of $v_x$ onto $\supp v \cap \G x$, which is compact.
\end{proof}

The formula \eqref{L2version1} below was stated and proved in \cite[(1)]{HoweL2}. Our argument includes the explicit formula \eqref{L2version3} for the projections on the isotypic components and the inverse \eqref{L2version2}. 

Our main tool to decompose $\omega$ will be Harish-Chandra's Plancherel formula (see e.g. 
\cite[Chapter 13]{WallachII}): for every $f\in C^\infty_c(\G)$,
\[
f(1)=\int_{\widehat{\G}} \Theta_\pi(f) \,d\mu(\pi)= \int_{\widehat{\G}} \Theta_{\pi^c}(f) \,d\mu(\pi)\,,
\]
where in the last equality we have used the invariance of the Plancherel measure with respect to taking contragredients (see e.g. \cite[Lemma 4.10(a)]{Fuhr}). 

\begin{cor}\label{L2version}
Let $\mu$ denote the Harish-Chandra Plancherel measure on $\G$.
For $\pi\in \widehat{\G}$, let $\L^2(\X)_{\pi\otimes\pi'}$ denote the Hilbert space associated with $\pi$ according to Corollary \ref{omega0-1993_thm3.1}. 
Then the restriction to $\G$ of the representation $(\omega_0, \mathcal \L^2(\X))$ decomposes as 
direct integral of Hilbert spaces
\begin{equation}\label{L2version1}
\L^2(\X)=\int_{\widehat{\G}}\mathcal \L^2(\X)_{\pi\otimes\pi'} \,d\mu(\pi)
\end{equation}
i.e.  for $v\in \mathcal{S}(\X)$,
\begin{equation}\label{L2version2}
v=\int_{\widehat{\G}} v_{\pi}   \,d\mu(\pi)
\end{equation}
where $v_\pi$ is defined by 
\begin{align}
\label{L2version3}
v_\pi(x)=\omega_0(\Theta_{\pi^c})v(x)&=
\int_{\G}\Theta_{\pi^c}(g)(\omega_0(g)v)(x) \, dg \notag\\
&=\int_\G\Theta_{\pi^c}(g) v(g^{-1}x) \, dg 
 \qquad (v\in C^\infty_c(\X^{\max}),\,  x\in \X^{\max})\,.
\end{align}
Also, for any $\Ca\in\mathcal U(\g)^\G$,
\begin{equation}\label{L2version4}
\omega_0(\Ca) v=\int_{\widehat{\G}}
\chi_\pi(\Ca)v_\pi \,d\mu(\pi)\qquad (v\in C_c^\infty(\X^{max}))\,,
\end{equation}
where $\chi_\pi: \mathcal U(\g)^\G\to\C$ is the infinitesimal character of $\pi$.
\end{cor}
\begin{proof}
The expression for $v_\pi$ is a consequence of Lemma \ref{L2version3}.
Harish-Chandra's Plancherel formula applied to the function $v_x$ of Lemma \ref{lemma:compact-support},(2), implies that for $v\in C_c^\infty(\X^{max})$ and $x\in\X^{max}$,
\[
\int_{\widehat{\G}} v_\pi(x) \, d\mu(\pi)
=\int_{\widehat{\G}}  \left[ \int_{\G} \Theta_{\pi^c}(g) v(g^{-1}x) \, dg\right] d\mu(\pi)
=\int_{\widehat{\G}} \Theta_{\pi^c}(v_x) \, d\mu(\pi)=v_x(1)=v(x)\,.
\]
By Theorem \ref{omega0-1993_thm3.1}, for every $u,v\in \mathcal{S}(\X)$, the inner product in $\L^2(\X)_{\pi\otimes\pi'}$ between $u_\pi=\omega(\Theta_{\pi^c})u$ and $v_\pi=\omega(\Theta_{\pi^c})v$ is 
\[
(u_\pi,v_\pi)_{\L^2(\X)_{\pi\otimes\pi'}}=(\omega(\Theta_{\pi^c})u,v)=(u,\omega(\Theta_{\pi^c})v)\,.
\]
Hence, for $u,v\in C_c^\infty(\X^{\max})$,
\begin{align*}
&\int_{\widehat{\G}} (u_\pi,v_\pi)_{\L^2(\X)_{\pi\otimes\pi'}}=
\int_{\widehat{\G}} (\omega(\Theta_{\pi^c})u,v) \, d\mu(\pi)\\
&= \int_{\widehat{\G}} \int_{\X^{\max}} u_\pi(x)\overline{v(x)} \, dx \, 
d\mu(\pi)= \int_{\X^{\max}}  \left[ \int_{\widehat{\G}} u_\pi(x)\, d\mu(\pi)\right] \overline{v(x)}\, dx\\
&= \int_{\X^{\max}} u(x)\overline{v(x)} \, dx = (u,v)_{\L^2(\X)}
\end{align*}
This verifies \eqref{L2version2} and \eqref{L2version3}. The statement \eqref{L2version4} is obvious.
\end{proof}

\begin{rem}
The algebra $\mathcal{U}(\g)^\G$ is a subalgebra of $\mathcal{Z}(\g)=\mathcal{U}(\g)^\g$. 
It agrees with $\mathcal{Z}(\g)$ when $\G$ is a real form of $\GL$, $\Sp$ or $\Og_{2p+1}$, but it is properly contained in $\mathcal{Z}(\g)$ when $\G$ is a real form of $\Og_{2p}(\C)$, such as $\Og_{p,p}$.
\end{rem}

\section{The pair $(\G, \G')=(\Sp_2(\R), \Og_{p,p})$, $p\geq 2$}
\label{section:SL2-Opp}
Here we continue the previous section for the example mentioned in the title. 
\subsection{Action of $\G$}
Let $\X=M_{2,p}(\R)$ be the space of matrices consisting of two rows of length $p\geq 2$ with real entries. We define an action $\omega_0$ of the group $\G$ on $\L^2(\X)$ as follows
\begin{equation}\label{C02}
\omega_0(g)v(x)=v(g^{-1}x) \qquad (g\in\G\,,\ v\in \L^2(\X)\,,\ x\in\X)\,.
\end{equation}
It is easy to check that this action preserves the $\L^2$-norm. This is the restriction to $\G$ of the Weil representation for that dual pair twisted by the character $\chi_+$, as in section \ref{section:Plancherel}.

\subsection{The Casimir elements and the Capelli operators}
The Lie algebra $\g$ is spanned by the elements $h$, $e^+$ and $e^-$ given in \eqref{hee}. 
We shall denote a matrix $x\in\X$ as 
\[
x=
\begin{pmatrix}
x_{1,1} & x_{1,2} & \cdots & x_{1,p}\\
x_{2,1} & x_{2,2}  & \cdots & x_{2,p}
\end{pmatrix}
\]
Then, by taking derivatives, we see that
\begin{align*}
\omega_0(h)&=\sum_{j=1}^p \left(x_{2,j}\partial_{x_{2,j}}-x_{1,j}\partial_{x_{1,j}}\right)\,,\\
\omega_0(e^+)&=-\sum_{j=1}^p x_{2,j}\partial_{x_{1,j}}\,,\\
\omega_0(e^-)&=-\sum_{j=1}^p x_{1,j}\partial_{x_{2,j}}\,.
\end{align*}
Then 
\begin{equation}\label{CasimirElement}
\Ca=h^2-2h+4e^+e^-\in \U(\g)^\G
\end{equation}
is the Casimir element. 
Let $\g'$ be the Lie algebra of $\Og_{p,p}$.
If $\Ca'\in \U(\g')^{\G'}$ is the Casimir element, then \cite[Ch. III, (2.3.4)]{HoweTan} implies that
\begin{equation}\label{Capelli100}
\omega_0(\Ca')=\omega_0(\Ca)-(p-1)^2+1\,.
\end{equation}
Formula \eqref{Capelli100} is one of Capelli's identities. We see from \cite[Ch. III, (2.3.4)]{HoweTan} that no translation of $\pm \omega_0(\Ca)$ is non-negative. Nevertheless, the study of resonances involve only the continuous part of the spectrum of an operator and, for $\omega_0(\Ca)$, this is a half-line. See Proposition \ref{Casimir's action}.

\subsection{Direct integral decomposition of the representation $(\omega_0, \L^2(\X))$ of $\Sp_2(\R)$}
We consider the following representations of $\G=\Sp_2(\R)=\SL_2(\R)$ (see e.g. \cite[p. 123]{LangSL2R} or \cite[\S 2.7]{knappLie2}):
\begin{enumerate}
\item
Discrete series representations $\Dsf^{n}$ and $\Dsf^{-n}$, $n\in \ZZ_{>0}$,
\item 
Spherical principal series representations $\pi_{0,\lambda}$, $\lambda\in \C$,
\item
Non-spherical principal series representations $\pi_{1,\lambda}$, $\lambda\in \C$.
\end{enumerate}
The discrete series $\Dsf^{\pm n}$ and the principal series $\pi_{\varepsilon,i\lambda}$ are unitary and irreducible for all pairs $(\varepsilon,\lambda)$ with $\varepsilon\in \{0,1\}$ and $\lambda\in \R$ except for $(\varepsilon,\lambda)=(1,0)$. In the latter case, $\pi_{1,0}=\Dsf^0_+\oplus\Dsf^0_-$ decomposes as the direct sum of two irreducible representations, $\Dsf^0_+$ and $\Dsf^0_-$, respectively called the holomorphic and anti-holomorphic limit of discrete series representations. 
The representation
$\pi_{0,i\lambda}$ is equivalent to $\pi_{0,i\lambda}^c=\pi_{0,-i\lambda}$, and $\pi_{1,i\lambda}$ is equivalent to $\pi_{1,i\lambda}^c=\pi_{1,-i\lambda}$. For the discrete series, we have 
$(\Dsf^{n})^c=\Dsf^{-n}$. Moreover, $(\Dsf^0_+)^c=\Dsf^0_-$.
The $\Dsf^{\pm n}$ with $n\in \ZZ_{>0}$, the $\pi_{\varepsilon,i\lambda}$ with $\varepsilon\in \{0,1\}$ and  $\lambda >0$, $\pi_{0,0}$, 
$\Dsf^0_+$, and $\Dsf^0_-$ are the irreducible tempered unitary representations of $\G$.

In these terms, Harish-Chandra's Plancherel formula reads as follows.
\begin{thm}\label{C0}
For any $f\in C_c^\infty(\G)$,
\begin{align*}
f(1)=\sum_{n=1}^\infty \Theta_{\Dsf^{n}}(f)&\frac{n}{2\pi}+\sum_{n=1}^\infty \Theta_{\Dsf^{-n}}(f)\frac{n}{2\pi}\\
&+\int_\R \Theta_{\pi_{0,i\lambda}}(f)\, \frac{\lambda}{8\pi}\tanh(\frac{\pi\lambda}{2})\,d\lambda
+ \int_\R \Theta_{\pi_{1,i\lambda}}(f)\, \frac{\lambda}{8\pi}\coth(\frac{\pi\lambda}{2})\,d\lambda\,.
\end{align*}
Here 
\begin{equation}
\label{even}
\Theta_{\pi_{0,i\lambda}}=\Theta_{\pi_{0,-i\lambda}} \qquad \text{and}  \qquad\Theta_{\pi_{1,i\lambda}}=\Theta_{\pi_{1,-i\lambda}}\,.
\end{equation}
\end{thm}
The Plancherel measure $\mu$ is given by:
\begin{align*}
d\mu(\Dsf^{n})&=d\mu(\Dsf^{-n})=\frac{n}{2\pi}\qquad (n\in \ZZ_{>0})\,,\\
d\mu(\pi_{0,i\lambda})&=\frac{\lambda}{8\pi}\tanh(\frac{\pi\lambda}{2})\,d\lambda\qquad(\lambda\in\R)\,,\\
d\mu(\pi_{1,i\lambda})&=\frac{\lambda}{8\pi}\coth(\frac{\pi\lambda}{2})\,d\lambda\qquad(\lambda\in\R)\,.
\end{align*}
In this section, the set $\X^{\max}\subseteq \X$ of matrices of maximal rank consist of matrices of rank equal $2$. 

Besides irreducible unitary representations, our computations will lead us to consider 
non-unitary principal series representations $\pi_{\varepsilon,\lambda}$ with 
$\varepsilon\in \{0,1\}$ and $\lambda\in \C$. They still have distribution character $\Theta_\pi$ which can be represented by a locally integrable function.
For such representations $\pi$, we define
\begin{equation}
\label{C1-omega0} 
\omega_0(\Theta_{\pi^c}): C_c^\infty(\X^{\max}) \ni u \mapsto u_\pi \in C^\infty(\X^{\max}) \subseteq C_c^\infty(\X^{\max})^*\,,
\end{equation}
where
\begin{equation}
\label{upi-2}
u_\pi(x)=\int_\G \Theta_{\pi^c}(g) u(g^{-1}x)\, dg \qquad (x\in\X^{\max})\,.
\end{equation}
Since the function $u$ is compactly supported, $u_\pi \in C^\infty(\X^{\max})$. Therefore it defines 
a distribution. Moreover, the integral \eqref{upi-2} is absolutely convergent.
Notice that, because of the growth of $\Theta_{\pi^c}$, this integral might not converge if $u,v\in \mathcal{S}(\X)$ 
and $\pi$ is not unitary.

Both spaces $C_c^\infty(\X^{\max})$ and $C_c^\infty(\X^{\max})^*$ are $\G$-modules, via the action by $\omega_0$, \eqref{C02}, and $\omega_0(\Theta_{\pi^c})$ is a $\G$-intertwining map. 

\subsection{The resonances of the Capelli operator}
\label{subsection:resonances}

Let $\nu$ denote the restriction of the Plancherel measure to the unitary principal series,
\begin{align*}
d\nu(\Dsf^{n})&=d\nu(\Dsf^{-n})=0\qquad (n\in \ZZ_{>0})\,,\\
d\nu(\pi_{0,i\lambda})&=\frac{\lambda}{8\pi}\tanh(\frac{\pi\lambda}{2})\,d\lambda\qquad(\lambda\in\R)\,,\\
d\nu(\pi_{1,i\lambda})&=\frac{\lambda}{8\pi}\coth(\frac{\pi\lambda}{2})\,d\lambda\qquad(\lambda\in\R)\,.
\end{align*}
In terms of Corollary \ref{L2version}, set
\begin{align*}
\L^2(\X)_{\rm cont}&=\int_{\hat \G} \L^2(\X)_{\pi}\,d\nu(\pi)\notag\\
&=\int_\R \L^2(\X)_{\pi_{0,i\lambda}}\frac{\lambda}{8\pi}\tanh(\frac{\pi\lambda}{2})\,d\lambda+
\int_\R \L^2(\X)_{\pi_{1,i\lambda}}\frac{\lambda}{8\pi}\coth(\frac{\pi\lambda}{2})\,d\lambda\,.
\end{align*}
\begin{pro}\label{Casimir's action}
The Casimir element $\mathcal C$, \eqref{CasimirElement}, acts on the principal series representation $\pi_{\varepsilon,i\lambda}$ via multiplication by $-\lambda^2-1$.
\end{pro}
\begin{prf}
This follows from \cite[pages 119 and 195]{LangSL2R}.
\end{prf}
Set $\Ca^+=-\omega_0(\Ca)-1$. Proposition \ref{Casimir's action} implies that the spectrum of $\Ca^+$ viewed as a densely defined operator on the Hilbert space $\L^2(\X)_{\rm cont}$ is equal to $[0,\infty)$. Hence the resolvent
\begin{equation}\label{C13}
(\Ca^+-z^2)^{-1}\in \mathcal{B}(\L^2(\X)_{cont})\qquad (z\in\C\,,\ \Im z>0)\,
\end{equation}
is well defined. Therefore, for $u,v\in C_c^\infty(\X^{\max})$ and $z$ as in \eqref{C13},
\begin{align}\label{resolvent}
(\Ca^+-z^2)^{-1}(u)(v)&=\int_\R (\lambda^2-z^2)^{-1}\left(\int_\X u_{\pi_{0,i\lambda}}(x)v(x)\,dx\right)\frac{\lambda}{8\pi}\tanh(\frac{\pi\lambda}{2})\,d\lambda\notag\\
&+\int_\R (\lambda^2-z^2)^{-1}\left(\int_\X u_{\pi_{1,i\lambda}}(x)v(x)\,dx\right)\frac{\lambda}{8\pi}\coth(\frac{\pi\lambda}{2})\,d\lambda\,.
\end{align}

\begin{lem}
\label{lem:PW}
For $\varepsilon\in \{0,1\}$ and $u,v\in C_c^\infty(\X^{\max})$,
\begin{equation}
\label{fpm}
f_{\varepsilon}(\lambda)=\int_\X u_{\pi_{\varepsilon,i\lambda}}(x)v(x)\,dx
\end{equation}
is an even Paley-Wiener type function of $\lambda\in\C$. 
\end{lem}
\begin{proof}
By Lemma \ref{lemma:compact-support}, the functions $g\mapsto u(gx)$ for $x\in \X^{\max}$ and 
\[
\psi(g)=\int_\X u(gx)v(x)\,dx \qquad (g\in\G)
\]
are smooth and compactly supported.
Hence,
\begin{align*}%\label{C15}
f_{\varepsilon}(\lambda)&=\int_\X \int_\G \Theta_{\pi_{\varepsilon,i\lambda}}(g^{-1}) u(g^{-1}x)v(x)\,dg\,dx\\
&=\int_\G \Theta_{\pi_{\varepsilon,i\lambda}}(g) \int_\X u(gx)v(x)\,dx\,dg
=\int_\G \Theta_{\pi_{\varepsilon,i\lambda}}(g) \psi(g)\,dg\,.
\end{align*}
 Let
\[
\A=\left\{
h_a=\begin{pmatrix}
a & 0\\
0 & a^{-1}
\end{pmatrix}, a>0\right\}
\]
and set
\[
\rho \left(
h_a\right)=a\,, \qquad 
D \left(
h_a\right)=a-a^{-1}\,.
\]
By \cite[VII, Theorem 4 and Corollary]{LangSL2R},
\begin{align}\label{C16}
\int_\G \Theta_{\pi_{\varepsilon,i\lambda}}(g) \psi(g)\,dg 
&= \frac{(-1)^\varepsilon}{2}\int_\A (\rho(h_a)^{i\lambda}+\rho(h_a)^{-i\lambda}) |D(h_a)|\int_{\G/\A} \psi(gh_ag^{-1})\,d\dot{g}\,dh_a\notag\\
&= (-1)^\varepsilon \int_\A \rho(h_a)^{i\lambda} |D(h_a)|\int_{\G/\A} \psi(gh_ag^{-1})\,d\dot{g}\,dh_a\,,
\end{align}
where $d\dot{g}$ is the invariant measure on $\G/\A$ such that $dg=d\dot{g} \, dh_a$ and where in the last equality we used the invariance of expression \eqref{C16} under the transformation $h_a\mapsto h_{a^{-1}}=h_a^{-1}$.
Since the Harish-Chandra orbital integral
\[
|D(h_a)|\int_{\G/\A} \psi(gh_ag^{-1})\,d\dot{g}
\]
is a smooth compactly supported function on $\A$, the claim follows. 

The evenness is an immediate consequence of $\eqref{even}$.
\end{proof}

Now we look for resonances.

\begin{lem}
\label{lem:shift}
Keep the notation of Lemma \ref{lem:PW} and let $L>0$ be a non-integer. 
\begin{enumerate}
\item
For every $z\in\C$ such that $\Im z>0$:
\begin{align}
\label{shift-tanh}
\int_\R &\frac{1}{\lambda^2-z^2} f_0(\lambda)\lambda\tanh(\frac{\pi\lambda}{2})\,d\lambda\notag\\
&=\int_{\R+iL} \frac{1}{\lambda+z}f_0(\lambda)\tanh(\frac{\pi\lambda}{2})\,d\lambda
+4i \sum_{\substack{k\in \ZZ \\ 0<2k+1<L}}\frac{1}{(2k+1)i+z} f_0((2k+1)i)\,.
\end{align}
\item 
For every $z\in\C$ such that $0<\Im z<1$:
\begin{align}
\label{shift-coth}
\int_\R &\frac{1}{\lambda^2-z^2} f_1(\lambda)\lambda \coth(\frac{\pi\lambda}{2})\,d\lambda\notag\\
&=\int_{\R+i} \frac{1}{\lambda-z}f_1(\lambda)\coth(\frac{\pi\lambda}{2})\,d\lambda
+\int_{\R+iL} \frac{1}{\lambda+z}f_1(\lambda)\coth(\frac{\pi\lambda}{2})\,d\lambda + F_L(z)\notag\\
&+\frac{2i}{z} \, f_1(0) 
+4i \sum_{\substack{k\in \ZZ\\0<2k<L}}\frac{1}{2ki+z}\, f_1(2ki)\,, \hskip 1cm\null
\end{align}
where $F_L$ is holomorphic for $-L<\Im z<1$.
\end{enumerate}
\end{lem}
\begin{proof}
Observe that 
\begin{equation}
\label{fraction}
\frac{2\lambda}{\lambda^2-z^2}=\frac{1}{\lambda-z}+\frac{1}{\lambda+z}\,.
\end{equation}
Then
\[
\int_\R (\lambda^2-z^2)^{-1}f_0(\lambda) \lambda\tanh(\frac{\pi\lambda}{2})\,d\lambda
=\frac{1}{2} \int_\R \left(\frac{1}{\lambda-z}+\frac{1}{\lambda+z}\right)f_0(\lambda)\tanh(\frac{\pi\lambda}{2})\,d\lambda\,.
\]
Also, since the hyperbolic tangent is an odd function and $f_0(\lambda)=f_0(-\lambda)$ by Lemma \ref{lem:PW}, 
\[
\int_\R \frac{1}{\lambda-z}f_0(\lambda)\tanh(\frac{\pi\lambda}{2})\,d\lambda
=\int_\R \frac{1}{\lambda+z}f_0(\lambda)\tanh(\frac{\pi\lambda}{2})\,d\lambda
\,.
\]
Therefore
\[
\int_\R (\lambda^2-z^2)^{-1}f_0(\lambda) \lambda\tanh(\frac{\pi\lambda}{2})\,d\lambda
=\int_\R \frac{1}{\lambda+z}f_0(\lambda)\tanh(\frac{\pi\lambda}{2})\,d\lambda\,.
\]
Since $f_0$ is of Paley-Wiener type, shifting the domain of integration to $\R+iL$ and the residue theorem yield \eqref{shift-tanh} 
because
\[
\Res_{\lambda=(2k+1)i} \tanh(\frac{\pi\lambda}{2})=\frac{2}{\pi}\,.
\]

The shifting argument above must be modified for the integral involving the hyperbolic cotangent because it has a pole at $\lambda=0$. So, we first shift the contour of integration from $\R$ to $\R+i$. We do not cross any singularity of $\lambda \coth\left(\frac{\pi\lambda}{2}\right)$ but the integrand has a simple pole at $\lambda=z$, which satisfies $0<\Im z<1$. The residue theorem
gives
\begin{equation}
\label{coth-1}
\int_\R \frac{1}{\lambda^2-z^2} f_1(\lambda)\lambda\coth(\frac{\pi\lambda}{2})\,d\lambda
=\int_{\R+i} \frac{1}{\lambda^2-z^2} f_1(\lambda)\lambda\coth(\frac{\pi\lambda}{2})\,d\lambda
+\pi i f_1(z) \coth(z)\,.
\end{equation}
We now apply \eqref{fraction} to the first term in \eqref{coth-1} and obtain 
\begin{align}
\label{coth-2}
\int_{\R+i} &\frac{1}{\lambda^2-z^2} f_1(\lambda)\lambda\coth(\frac{\pi\lambda}{2})\,d\lambda\notag\\
&=\frac{1}{2}  \int_{\R+i} \frac{1}{\lambda-z} f_1(\lambda)\coth(\frac{\pi\lambda}{2})\,d\lambda
+ \frac{1}{2}  \int_{\R+i} \frac{1}{\lambda+z} f_1(\lambda)\coth(\frac{\pi\lambda}{2})\,d\lambda\,.
\end{align}
Notice that, on the right-hand side of \eqref{coth-2}, the first integral defines a holomorphic function for $\lambda \notin \R+i$, whereas the second defines a holomorphic function for $\lambda \notin \R-i$. We therefore shift the domain of integration of the second integral and apply the residue theorem again. Since
\[
\Res_{\lambda=2ki} \coth(\frac{\pi\lambda}{2})=\frac{2}{\pi}\,,
\]
this yields (forgetting for a moment the constant $\frac{1}{2}$):
\begin{align*}
\int_{\R+i}& \frac{1}{\lambda+z} f_1(\lambda)\coth(\frac{\pi\lambda}{2})\,d\lambda=\notag\\
&\int_{\R+iL} \frac{1}{\lambda+z} f_1(\lambda)\coth(\frac{\pi\lambda}{2})\,d\lambda
+4i \sum_{\substack{k\in \ZZ\\0<2k<L}}\frac{1}{2ki+z}\, f_1(2ki)\,.
\end{align*}
On the other hand, since $f_1$ is even,
\begin{align*}
i\pi f_1(z)\coth z&=i\pi \sum_{\substack{k\in \ZZ\\0\leq 2k<L}} \Res_{z=-2ki} [f_1(z)\coth z] \frac{1}{2ki+z} + F_L
\\&=
2i \sum_{\substack{k\in \ZZ\\0\leq 2k<L}} \frac{1}{2ki+z}\, f_1(2ki) +F_L\,,
\end{align*}
where $F_L$ is holomorphic for $-L<\Im z<1$. 
By substituting all these expressions in \eqref{coth-1} we then obtain \eqref{shift-coth}.
\end{proof}

\begin{thm}
\label{thm:resonancesCapelli-SL2R}
Considered as a $C^\infty_c(\X^{\max})^*$-valued linear operator on $C^\infty_c(\X^{\max})$, the resolvent $(\Ca^+-z^2)^{-1}$ extends from the upper half-plane $\C^+$ to a meromorphic function on $\C$, with simples poles (the resonances of $\Ca^+$) at $z=-in$ with $n\in \ZZ_{\geq 0}$. 

The residue operator at the resonance $z=-2ki$ is
\[
\omega_0(\Theta_{\pi_{1,2k}^c}): C^\infty_c(\X^{\max}) \ni u \longrightarrow u_{\pi_{1,2k}} 
\in C^\infty(\X^{\max})\subseteq  C^\infty_c(\X^{\max})^*, 
\]
where
\[
u_{\pi_{1,2k}}(x)=\int_\G \Theta_{\pi_{1,2k}^c}(g)u(g^{-1}x)\, dg \qquad (x\in \X^{\max})\,.
\]
The residue operator at the resonance $z=-(2k+1)i$ is
\[
\omega(\Theta_{\pi_{0,2k+1}^c}): C^\infty_c(\X^{\max}) \ni u \longrightarrow u_{\pi_{0,2k+1}} 
\in C^\infty(\X^{\max})\subseteq  C^\infty_c(\X^{\max})^*, 
\]
where
\[
u_{\pi_{0,2k+1}}(x)=\int_\G \Theta_{\pi_{0,2k+1}^c}(g)u(g^{-1}x)\, dg   \qquad (x\in \X^{\max})\,.
\]
\end{thm}
\begin{proof}
This is an immediate consequence of \eqref{resolvent}, since 
\begin{align*}
(\Ca^+-&z^2)^{-1}(u)(v)\\
&=\frac{1}{8\pi}\int_\R (\lambda^2-z^2)^{-1} f_0(\lambda)\lambda\tanh(\frac{\pi\lambda}{2})\,d\lambda
+\frac{1}{8\pi} \int_\R (\lambda^2-z^2)^{-1} f_1(\lambda) \lambda\coth(\frac{\pi\lambda}{2})\,d\lambda\,,
\end{align*}
where $f_\varepsilon(\lambda)$ are defined from the fixed $u,v\in C^\infty_c(\X^{\max})$ according to \eqref{fpm}.
Observe that integrals over $\R+iL$ in \eqref{shift-tanh} and \eqref{shift-coth} are holomorphic 
on $\Im z>-L$, whereas the integral over $\R+i$ in \eqref{shift-coth} is holomorphic 
on $\Im z<1$. Since $L>0$ is an arbitrary non-integer, the required meromorphic extension follows. Up to the constant $\frac{i}{2\pi}$ (or $\frac{i}{4\pi}$ when $n=0$), which does not play any special role and we will ignore, the residue operator $R_n$ at $z=-in$ 
maps $u$ into the distribution $R_nu$ such that $(R_nu)(v)$ is the residue of $(\Ca^+-z^2)^{-1}(u)(v)$ at $-in$, i.e. $f_0(2ki)$ if $n=2k$ and $f_1((2k+1)i)$ if $n=2k+1$.
\end{proof}

\subsection{The residue representations}

In this section we study the images of the residue operators, namely $\omega_0(\Theta_{\pi_{1,2k}})(C^\infty_c(\X^{\max}))$ and $\omega_0(\Theta_{\pi_{0,2k+1}})(C^\infty_c(\X^{\max}))$, where $k\in \ZZ_{\geq 0}$, as $\G$-spaces. 

We have observed in subsection \ref{subsection:resonances} that the images of the residue operators are spaces of distributions on $\X^{\max}$ and their elements are in fact in $C^\infty(\X^{\max})$. We first show that they are not only subspaces of $C^\infty_c(\X^{\max})^*$, but of $S(\X)^*$, the space of tempered distributions on $\X$.

\begin{lem}
\label{lem:g-gx}
Let us view $\R^p$ as a real Hilbert space with norm defined by the dot product. We identify the space of $m\times n$ matrices $\M_{m,n}(\R)$ with $\Hom(\R^n,\R^m)$, where the matrix $x$ sends a column vector $v\in\R^n$ to the column vector $xv\in \R^m$. Denote by $|x|$ the operator norm of $x$. Assume $m\leq n$ and let $\M_{m,n}^{max}(\R)\subseteq \M_{m,n}(\R)$ be the subset of matrices of maximal rank ($=m$). Then for any compact subset $\E\subseteq \M_{m,n}(\R)$ there is a constant $0<C<\infty$ such that
\[
|g|\leq C|gx| \qquad (g\in \M_{m,m}(\R)\,,\ x\in \E)\,.
\]
\end{lem}
\begin{proof}
For each $x\in \M_{m,n}^{max}(\R)$ there is $k_x\in \Og_n$ such that
\[
xk_x=(y_x,0)\,,
\]
where $y_x\in \GL_m(\R)$. Moreover the map $x\mapsto y_x$ is continuous. Hence 
\[
0<C=\max_{x\in \E}|y_x^{-1}|<\infty\,.
\]
Thus, for every $x\in \E$,
\[
|g|=|gy_xy_x^{-1}|\leq |gy_x| |y_x^{-1}|\leq |gy_x| C= C|gxk_x|=C|gx|\,.
\]
\end{proof}

The following proposition holds for every principal series representation. 

\begin{pro}
\label{pro:maps-tempered}
For every $\varepsilon\in \{0,1\}$ and $\lambda\in \C$, 
\begin{equation}
\label{is-tempered}
\omega(\Theta_{\pi_{\varepsilon,i\lambda}^c})(C^\infty_c(\X^{\max})) \subset \Ss(\X)^*.
\end{equation}
\end{pro}
\begin{proof}
Lemma \ref{lem:g-gx} implies that for every fixed $u\in C_c^\infty(\X^{max})$ and $N>0$ there is a seminorm $q_{N,u}$ on the space $\Ss(\X)$ such that
\[
\int_\X |u(x)v(gx)|\, dx \leq q_{N,u}(v) (1+|g|)^{-N} \qquad (g\in \G\,,\ v\in \Ss(\X))\,.
\]
Recall the notation \eqref{C16} and let
\[
\N=\left\{n_r=\left(
\begin{array}{ccc}
1 & r\\
0 & 1
\end{array}\right), r\in\R\right\}\,.
\]
Let $|\cdot|_{\rm HS}$ denote the Hilbert--Schmidt norm. Then
\[
|h_a n_r|_{\rm HS}^2=
\left|\left(
\begin{array}{ccc}
a & ar\\
0 & a^{-1}
\end{array}\right)\right|_{\rm HS}^2
=a^2+a^{-2}+(ar)^2 \qquad (h_a\in \A\,,\ n_r\in \N)\,.
\]
Hence there is a constant $C_N$ such that
\begin{align*}
\rho(h_a)\int_\N(1+|h_a n_r|)^{-N}\,dn_r&\leq C_N\rho(h_a)\int_\N(1+|h_a n_r|_{\rm HS}^2)^{-N/2}\,dn_r\\
&\leq C_N\rho(h_a)\int_\R(1+a^2+a^{-2})^{-N/4}(1+(ar)^2)^{-N/4}\,dr\\
&=\left(C_N\int_\R(1+r^2)^{-N/4}\,dr\right) (1+a^2+a^{-2})^{-N/4}\,.
\end{align*}
For any fixed $\lambda\in\C$,
\begin{align*}
\int_\G|&\Theta_{\pi_{\varepsilon,i\lambda}^c}(g)|(1+|g|)^{-N}\,dg\\
& =\int_\A|\Theta_{\pi_{\varepsilon,i\lambda}^c}(h_a)| |D(h_a)|^2\int_{\G/\A}(1+|g^{-1}h_a g|)^{-N}\,d\overset . g\,dh_a\\
& \leq  \int_\A(|\rho(h_a)^{i\lambda}|+|\rho(h_a)^{-i\lambda}|) |D(h_a)|\int_{\G/\A}(1+|g^{-1}h_a g|)^{-N}\,d\overset . g\,dh_a\\
& =\int_\A(|\rho(h_a)^{i\lambda}|+|\rho(h_a)^{-i\lambda}|) \rho(h_a)\int_\N(1+|h_a n_r|)^{-N}\,dn_r\,dh_a\\
&\leq \left(C_N\int_\R(1+r^2)^{-N/4}\,dr\right) \int_0^\infty 2(1+e^{2t}+e^{-2t})^{-N/4}\sinh2t\,dt\,,
\end{align*}
where $e^t=a$ and we used integration formula \cite[VII, INT2]{LangSL2R}. The expression above is finite for $N>0$ large enough. Hence there is a seminorm $q_u$, which depends on 
$\Theta_{\pi_{\varepsilon,i\lambda}^c}$, on the space $\Ss(\X)$ such that
\[
\int_\G\int_\X|\Theta_{\pi_{\varepsilon,i\lambda}}^c(g) u(g^{-1}x)v(x)|\,dx\,dg\leq q_u(v)
 \qquad (u\in C_c^\infty(\X^{max}), v\in \Ss(\X))\,.
\]
This verifies \eqref{is-tempered}. 
\end{proof}

For $m\in \ZZ$, and for fixed $\lambda\in\C$, let $\varphi_m$ be the function on $\G$ defined in terms of the Iwasawa decomposition $\G=\K\A\N$ by 
\[
\varphi_m\left( k_\theta h_a n_r
\right)
=\rho(h_a)^{-(\lambda+1)} e^{im\theta}\,,
\]
where 
$k_\theta=\begin{pmatrix}
\cos\theta & \sin\theta\\
-\sin\theta & \cos\theta
\end{pmatrix}
 \in \K=\SO_2(\R)$. 
Then the $(\mathfrak{g},\K)$-module of $\pi_{\varepsilon,\lambda}$ is 
\[
\V_{\varepsilon,\lambda,\K}=\bigoplus_{\substack{m\in\ZZ \\m\equiv \varepsilon}} \C \varphi_m\,,
\]
where $m\equiv \varepsilon$ means $m-\varepsilon \in 2\ZZ$.

We now focus on the case $\pi_{\varepsilon,n}$, where $\varepsilon\in \{0,1\}$, $n\in \ZZ_{\geq 0}$ and $n \not\equiv \varepsilon$, as in Theorem \ref{thm:resonancesCapelli-SL2R}.
These representations are all reducible:
\begin{enumerate}
\item
$\pi_{1,0}=\Dsf^0_+\oplus \Dsf^0_-$ decomposes into the holomorphic and anti-holomorphic limits of discrete series representations. The $(\g,\K)$-modules of $\Dsf^0_-$ and $\Dsf^0_+$
are respectively
\[
\bigoplus_{\substack{m<0 \\m\equiv 1}} \C \varphi_m
\qquad \text{and} \qquad
\bigoplus_{\substack{m>0 \\m\equiv 1}} \C \varphi_m\,.
\]
\item
For all $(\varepsilon,n)$ where $\varepsilon\in \{0,1\}$, $n\in \ZZ_{>0}$ and 
$n \not\equiv \varepsilon$, the $(\g,\K)$-module $\V_{\varepsilon,n,\K}$ contains two irreducible submodules:
\[
\V_{\K}^{-n}=\bigoplus_{\substack{m< -n \\ m\equiv \varepsilon}} \C \varphi_m
\qquad\text{and}\qquad
\V_{\K}^{n}=\bigoplus_{\substack{m> n \\ m\equiv \varepsilon}} \C \varphi_m\,.
\]
$\V_{\K}^{-n}$ and $\V_{\K}^{n}$ are isomorphic to the $(\g,\K)$-modules of the discrete series representations $\Dsf^{-n}$ and $\Dsf^{n}$, respectively. 
The quotient module 
\[
\V_n=\V_{\varepsilon,n,\K}/ (\V_{\K}^{-n}+\V_{\K}^{n}) = 
\bigoplus_{\substack{-n\leq m\leq n \\m\equiv \varepsilon}} \C \varphi_m
\]
is finite dimensional, of dimension $n$, and isomorphic to its contragredient. 
\end{enumerate}

The above composition series show that 
\begin{align*}
&\Theta_{\pi_{1,0}}=\Theta_{\Dsf^0_-}+\Theta_{\Dsf^0_+}\\
&\Theta_{\pi_{\varepsilon,n}}=\Theta_{\Dsf^{-n}}+\Theta_{\V_n}+\Theta_{\Dsf^{n}}
\qquad \text{$(\varepsilon\in \{0,1\}$, $n\in \ZZ_{>0}$ and 
$n \not\equiv \varepsilon$)}\,.
\end{align*}

Hence, in the notation of Theorem \ref{thm:resonancesCapelli-SL2R}, 
\begin{align}
\label{Theta-forms1}
&\omega_0(\Theta_{\pi_{0,1}^c})=\omega_0(\Theta_{(\Dsf^0_-)^c})+\omega_0(\Theta_{(\Dsf^0_+)^c})\\
\label{Theta-forms2}
&\omega_0(\Theta_{\pi_{\varepsilon,n}^c})=\omega_0(\Theta_{(\Dsf^{-n})^c})+\omega_0(\Theta_{\V_n^c})+\omega_0(\Theta_{(\Dsf^{n})^c}) \notag \\
&\null\hskip 6cm\qquad \text{$(\varepsilon\in \{0,1\}$, $n\in \ZZ_{>0}$ and 
$n \not\equiv \varepsilon$)}\,.\quad\null
\end{align}

Proposition \ref{pro:maps-tempered} extends to each of the above subquotients 
of $\pi_{\varepsilon,n}$.

\begin{pro}
\label{prop:is-tempered}
Let $\pi \in \{\Dsf^0_-, \Dsf^0_+, \Dsf^{-n}, \V_n, \Dsf^{n}\}$ be a subquotient of $\pi_{\varepsilon,n}$, where $(\varepsilon,n)\in \{0,1\}\times \ZZ_{\geq 0}$ and $n \not\equiv \varepsilon$. Then
\[
\omega_0(\Theta_{\pi^c})(C^\infty_c(\X^{\max})) \subset \Ss(\X)^*.
\]
\end{pro}
\begin{proof}
Since $\Dsf^0_-, \Dsf^0_+, \Dsf^{-n}$ and $\Dsf^{n}$ are tempered unitary representations, the property holds for them (even with $\Ss(\X)$ instead of $C^\infty_c(\X^{\max})$) by Corollary 
\ref{omega0-1993_thm3.1}. We only have to consider the case where $\pi$ is the finite dimensional 
subquotient acting on $\V_n$. Because of the formula for the restriction to $\A$ of $\Theta_{\V_n^c}$, see e.g. \cite[VII, Lemma 2]{LangSL2R}, exactly the same proof as Proposition \ref{pro:maps-tempered} applies in this case as well.
\end{proof}

Let $\proj_m$ denote the projection of  $(\omega_0|_\G,C^\infty_c(\X^{\max}))$ onto its isotypic component of type $\chi_m$, where $\chi_m(k_\theta)=e^{im\theta}$, i.e. 
\[
\proj_m f(x)=\int_\K \chi_m(k) f(k^{-1}x) \, dk \qquad (f\in C^\infty_c(\X^{\max}), x\in \X^{\max})\,.
\]
In other words, $\proj_m=\omega_0(\chi_m\, dk)$.
Denote by $\proj^\G_m$ the projection of principal series representation $\pi_{\varepsilon,n}$ onto its 
its isotypic component of type $\chi_m$, i.e. 
\[
\proj^\G_m \varphi(g)=\int_\K \chi_m(k) \varphi(k^{-1}g) \, dk \qquad (\varphi \in\V_{\varepsilon,n,\K},\, g\in \G)\,,
\]
i.e. $\proj^\G_m =\pi_{\varepsilon,n}(\chi_m\, dk)$.

Moreover, for every $n\in \ZZ_{\geq 0}$, let $\varepsilon\in \{0,1\}$ such that $\varepsilon \not\equiv n$. Set 
\begin{equation}
\label{projections}
\proj_{n,<}=\bigoplus_{\substack{m< -n \\m\equiv \varepsilon}} \proj_m\,,
\qquad
\proj_{n,{\rm fin}}=\bigoplus_{\substack{-n\leq m\leq n \\m\equiv \varepsilon}} \proj_m\,,
\qquad
\proj_{n,>}=\bigoplus_{\substack{m> n \\m\equiv \varepsilon}} \proj_m\,.
\end{equation}
By replacing $\proj_m$ with $\proj_m^\G$, we similarly define the projections $\proj_{n,<}^\G$, $\proj_{n,{\rm fin}}^\G$, and $\proj_{n,>}^\G$.

Let $u\in C^\infty_c(\X^{\max})$. Recall from Lemma \ref{lemma:compact-support} that 
$u_x:\G\to \C$, defined for $g\in \G$ by $u_x(g)=u(g^{-1}x)$, is in $C_c^\infty(\G)$. For $f:\G\to \C$, we set
$f^\vee(g)=f(g^{-1})$ for all $g\in \G$. To simplify notation, we will write $u_x^\vee$ instead
of $(u_x)^\vee$.
The following lemma links $\proj_n$ and $\proj^\G_n$ for such functions. Similar relations extend 
to their sums in \eqref{projections}.

\begin{lem}
\label{lem:proj}
Let $u\in C^\infty_c(\X^{\max})$. Then 
\[\proj_m^\G \left(u_x^\vee\right)=(\proj_m u)_x^\vee \qquad (x\in \X^{\max})\,.\]
\end{lem}
\begin{proof}
Since $u_x^\vee(k^{-1}g)=u_x(g^{-1}k)=u(k^{-1}gx)$, we have for every $g\in \G$
\[
\proj_m^\G\left(u_x^\vee\right)(g)=\int_\K \chi_m(k) u_x^\vee(k^{-1}g) \, dk
=\int_\K \chi_m(k) u(k^{-1}gx) \, dk=(\proj_m u)(gx)=(\proj_m u)_x^\vee(g)\,.
\]
\end{proof}

The following fact is well-known.
\begin{lem}
\label{lem:matrix-coefficient1}
For any subquotient $\pi$ of the principal series representation $\pi_{\varepsilon,n}$ of $\G$, we have
\begin{equation}
\label{matrix-coefficient1}
\int_\G (\pi(g)\varphi_m,\varphi_m)_{\L^2(\K)} f(g) \, dg=\int_\G  (\pi(g)\varphi_m,\varphi_m)_{\L^2(\K)} (\proj_m^\G f)(g) \, dg\,.
\end{equation}
\end{lem}
\begin{proof}
Replacing $g$ by $gk$ and integrating over $\K$, the left-hand side of the equality becomes 
\begin{align*}
\int_\K \int_\G (\pi(gk)\varphi_m,\varphi_m)_{\L^2(\K)} &f(gk) \, dg\, dk= 
\int_\K \int_\G (\pi(g)\pi(k)\varphi_m,\varphi_m)_{\L^2(\K)} f(gk) \, dg\, dk\\
&= 
\int_\G (\pi(g)\varphi_m,\varphi_m)_{\L^2(\K)} \left(\int_\K \chi_m(k^{-1}) f(gk) dk\right) dg\,,
\end{align*}
which is the right-hand side of \eqref{matrix-coefficient1}.
\end{proof}

\begin{lem}
\label{lem:proj1}
Let $\pi \in \{\Dsf^0_-, \Dsf^0_+, \Dsf^{-n}, \V_n, \Dsf^{n}\}$ be a subquotient of $\pi_{\varepsilon,n}$, where $(\varepsilon,n)\in \{0,1\}\times \ZZ_{\geq 0}$, and let $\proj$ denote the projection of the $(\mathfrak{g},\K)$-module $\V_{\varepsilon,n,\K}$ of $\pi_{\varepsilon,n}$ onto 
the $(\mathfrak{g},\K)$-module of $\pi$. 
Then
\[
(u,v)\longrightarrow (\omega_0(\Theta_{\pi^c})u,v)=\int_\X u_\pi(x)\overline{v(x)} \, dx=\int_{\X^{\max}} u_\pi(x)\overline{v(x)} \, dx\,,
\]
where $u_\pi$ is as in \eqref{upi-2}, 
is a hermitian bilinear form on $C^\infty_c(\X^{\max})$. Moreover, 
\[
(\omega_0(\Theta_{\pi^c})u,v)=
(\omega_0(\Theta_{\pi^c})\proj u,\proj v) \qquad (u,v\in C^\infty_c(\X^{\max}))\,.
\]
\end{lem}
\begin{proof}
The fact that $(\omega_0(\Theta_{\pi^c})u,v)$ is a bilinear hermitian form on $C^\infty_c(\X^{\max})$ -- and even on $\Sg(\X)$ -- when $\pi$ is $\Dsf^0_-, \Dsf^0_+, \Dsf^{-n}$, or $\Dsf^{n}$,  is part of Corollary \ref{omega0-1993_thm3.1}. For $\V_n$, this is a consequence of \cite[VII, \S 4, Lemmas 2 and 3]{LangSL2R}, which shows that $\Theta_{\V_n^c}$ is real valued. 

To prove the last statement, let us suppose for definiteness that $\pi=\Dsf^{n}$, 
so that $\proj=\proj_{n,>}$.
Notice that, for $u\in C^\infty_c(\X^{\max})$, we have $u(gx)=u_x(g^{-1})=u_x^\vee(g)$. 
Hence for every $v\in C^\infty_c(\X^{\max})$, by \eqref{C1-omega0} and \eqref{upi-2} and 
Lemmas \ref{lem:matrix-coefficient1} and \ref{lem:proj}, we obtain:
\begin{align*}
(\omega_0(\Theta_{(\Dsf^{n})^c})u,v)&=\int_{\X^{\max}}\int_\G \Theta_{(\Dsf^{n})^c}(g)u(g^{-1}x)\overline{v(x)} \, dg\, dx\\
&=\int_{\X^{\max}}\int_\G \Theta_{\Dsf^{n}}(g^{-1})u(g^{-1}x)\overline{v(x)} \, dg\, dx\\
&=\int_{\X^{\max}}\int_\G \Theta_{\Dsf^{n}}(g)u_x^\vee(g)\overline{v(x)} \, dg\, dx\\
&=\sum_{\substack{m> n \\m\equiv \varepsilon}} \int_{\X^{\max}}\int_\G 
(\Dsf^{n}(g)\varphi_m,\varphi_m)_{\L^2(\K)} u_x^\vee(g)\overline{v(x)} \, dg\, dx\,\\
&=\sum_{\substack{m> n \\m\equiv \varepsilon}}  \int_{\X^{\max}}\int_\G 
(\Dsf^{n}(g)\varphi_m,\varphi_m)_{\L^2(\K)} \proj_m^\G u_x^\vee(g)\overline{v(x)} \, dg\, dx\,\\
&= \int_{\X^{\max}}\int_\G 
(\Dsf^{n}(g)\varphi_m,\varphi_m)_{\L^2(\K)} \proj_{n,>}^\G u_x^\vee(g)\overline{v(x)} \, dg\, dx\,\\
&= \int_{\X^{\max}}\int_\G 
(\Dsf^{n}(g)\varphi_m,\varphi_m)_{\L^2(\K)} (\proj_{n,>} u)_x^\vee(g)\overline{v(x)} \, dg\, dx\,,
\end{align*}
which gives
$
(\omega_0(\Theta_{(\Dsf^{n})^c})u,v)=(\omega_0(\Theta_{(\Dsf^{n})^c}) \proj_{n,>} u,v)
$
via the same computations in reverse order.
The result now follows since the form is hermitian.
\end{proof}

Let us define hermitians forms on $C^\infty_c(\X^{\max})$ by 
\begin{align*}
(u,v)_{1,0}&=(\omega_0(\Theta_{\pi_{1,0}^c})u,v)\,,\\
(u,v)_{1,0,<}&=(\omega_0(\Theta_{(\Dsf^0_-)^c})u,v)\,,\\
(u,v)_{1,0,>}&=(\omega_0(\Theta_{(\Dsf^0_+)^c})u,v)\,,
\end{align*}
and, for $(\varepsilon,n)\in \{0,1\} \times \ZZ_{>0}$ with $n \not\equiv \varepsilon$,
\begin{align*}
(u,v)_{\varepsilon,n}&=(\omega_0(\Theta_{\pi_{\varepsilon,n}^c})u,v)\,,\\
(u,v)_{\varepsilon,n,<}&=(\omega_0(\Theta_{(\Dsf^{-n})^c})u,v)\,,\\
(u,v)_{\varepsilon,n,{\rm fin}}&=(\omega_0(\Theta_{\V_n^c})u,v)\,,\\
(u,v)_{\varepsilon,n,>}&=(\omega_0(\Theta_{(\Dsf^{n})^c})u,v)\,.
\end{align*}
Hence, by \eqref{Theta-forms1} and \eqref{Theta-forms2}, for all $u,v\in C^\infty_c(\X^{\max})$,
\begin{align}
\label{forms1}
(u,v)_{1,0}&=(u,v)_{1,0,<}+(u,v)_{1,0,>}\,,\\
\label{forms2}
(u,v)_{\varepsilon,n}&=(u,v)_{\varepsilon,n,<}+(u,v)_{\varepsilon,n,{\rm fin}}+(u,v)_{\varepsilon,n,>}\,.
\end{align}
By Lemma \ref{lem:proj1}, these forms agree with their restrictions to the corresponding projections 
$\proj C^\infty_c(\X^{\max})$. 
For every $(\varepsilon,n)\in \{0,1\} \times \ZZ_{\geq 0}$ with $n \not\equiv \varepsilon$, set
\[
C^\infty_{n,<}=\proj_{n,<} \left(C^\infty_c(\X^{\max})\right)\,, \qquad 
C^\infty_{n,{\rm fin}}=\proj_{n,{\rm fin}} \left(C^\infty_c(\X^{\max})\right)\,, \qquad 
C^\infty_{n,>}=\proj_{n,>} \left(C^\infty_c(\X^{\max})\right)\,.
\]

Let 
\[
\rad_{1,0}\,, \quad \rad_{1,0,<}\,, \quad  \rad_{1,0,>}\,,\quad \rad_{\varepsilon,n}\,,\quad \rad_{\varepsilon,n,<}\,, \quad \rad_{\varepsilon,n,{\rm fin}}\,, \quad \rad_{\varepsilon,n,>}\,.
\]
respectively denote the radicals of the forms in \eqref{forms1} and \eqref{forms2} as forms on $C^\infty_c(\X^{\max})$.

We will treat in the following the cases corresponding to $n\in \ZZ_{>0}$, the case for $n=0$ being similar (and easier). 

\begin{lem}
Let $(\varepsilon,n)\in \{0,1\} \times \ZZ_{>0}$ with $n \not\equiv \varepsilon$. 
Then, corresponding to the direct sum decomposition with respect to the action of $\K$
\begin{equation}
\label{Ktypes1}
C^\infty_c(\X^{\max})=C^\infty_{n,<}\oplus 
C^\infty_{n,{\rm fin}} \oplus
C^\infty_{n,>}\,,
\end{equation}
we have 
\begin{equation}
\label{Ktypes2}
C^\infty_c(\X^{\max})/\rad_{\varepsilon,n}= 
C^\infty_{n,<}/\left({\rad_{\varepsilon,n,<}|_{C^\infty_{n,<}}}\right)\oplus 
C^\infty_{n,{\rm fin}}/\left({\rad_{\varepsilon,n,{\rm fin}}|_{C^\infty_{n,{\rm fin}}}}\right) \oplus
C^\infty_{n,>}/\left({\rad_{\varepsilon,n,>}|_{C^\infty_{n,>}}}\right)\,,
\end{equation}
where 
on the right-hand side we take the restriction of the considered forms to the ranges of the
corresponding projections and
\begin{align*}
&C^\infty_{n,<}/\left({\rad_{\varepsilon,n,<}|_{C^\infty_{n,<}}}\right)=C^\infty_c(\X^{\max})/\rad_{\varepsilon,n,<}\\
&C^\infty_{n,{\rm fin}}/\left({\rad_{\varepsilon,n,{\rm fin}}|_{C^\infty_{n,{\rm fin}}}}\right)=
C^\infty_c(\X^{\max})/\rad_{\varepsilon,n,{\rm fin}}\\
&C^\infty_{n,>}/\left({\rad_{\varepsilon,n,>}|_{C^\infty_{n,>}}}\right)=C^\infty_c(\X^{\max})/\rad_{\varepsilon,n,>}
\end{align*}
\end{lem}
\begin{proof}
The $\K$-type decomposition of $C^\infty_c(\X^{\max})$ in \eqref{Ktypes1} implies that 
\[
\rad_{\varepsilon,n}=\rad_{\varepsilon,n,<}\cap \rad_{\varepsilon,n,{\rm fin}} \cap 
\rad_{\varepsilon,n,>}\,.
\]
Since
\[
\rad_{\varepsilon,n,<}=
\left(\rad_{\varepsilon,n,<}|_{C^\infty_{n,<}}\right)\oplus 
C^\infty_{n,{\rm fin}}\oplus C^\infty_{n,>}
\]
and similarly for the other two, we obtain that 
\[
\rad_{\varepsilon,n}= 
\left(\rad_{\varepsilon,n,<}|_{C^\infty_{n,<}}\right)\oplus
\left(\rad_{\varepsilon,n,{\rm fin}}|_{C^\infty_{n,{\rm fin}}}\right)\oplus
\left(\rad_{\varepsilon,n,>}|_{C^\infty_{n,>}}\right)\,,
\]
from which \eqref{Ktypes2} follows.
\end{proof} 
 
To the first and the last quotient, which correspond to $\Dsf^{-n}$ and $\Dsf^{n}$, 
we can apply Theorem \ref{1993_thm3.1}. Hence the range of 
$\omega_0(\Theta_{(\Dsf^{\pm n })^c})$ is a $\G\cdot\G'$-module of the form 
\[
\Dsf^{\pm n}\otimes (\Dsf^{\pm n})'
\]
where $(\Dsf^{\pm n })'$ is a irreducible unitary (usually not tempered) $\G'$-module.

Theorem \ref{1993_thm3.1} does not apply to $\omega_0(\Theta_{\V_n^c})$ because the growth of the character of $\V_n$ does not allow to extend it to $\mathcal{S}(\X)$. Nevertheless, Proposition 
\ref{prop:is-tempered} ensures that $\omega_0(\Theta_{\V_n^c})$ is a $\G'$-module under 
$\omega_0$. Hence
\[
\omega_0(\Theta_{\V_n^c})=\V_n\otimes(\V_n)'
\]
where $(\V_n)'$ is an admissible, quasi-simple representation of $\G'$. For $n\geq 1$, understanding its structure would require work parallel to \cite{HT93} and we defer it to a future article. If $n=1$, then $\V_1$ is the trivial representation and  $(\V_1)'$ is irreducible and unitary.   
\medskip

We summarize these results as follows.
The Capelli operator $\mathcal{C}^+$ is an unbounded self-adjoint operator on $\L^2(\X)$. Its spectrum is the union of a continuous and a discrete part. We consider subspace $\L^2(\X)_{\rm cont} \subseteq \L^2(\X)$ on which the Capelli operator has a continuous spectrum. 
The operator $\mathcal{C}^+$ commutes with the action of $\Sp_2(\R)$ on $\L^2(\X)$ via
$\omega_0$. 
We consider the direct integral decomposition of $\L^2(\X)_{\rm cont}$ as $\Sp_2(\R)$-module under $\omega_0$. 
Each isotypic component is a multiple of a unitary principal spherical representation $\pi_{\varepsilon,i\lambda}$ of $\Sp_2(\R)$, where $\lambda\in \R$. 
The Capelli operator acts on each of them as scalar multiplication. 
As a bounded operator on $\L^2(\X)_{\rm cont}$, the resolvent $(\mathcal{C}^+-z^2)^{-1}$ is defined for $z$ in the upper half-plane $\C^+$. Its restriction  
$(\mathcal{C}^+-z^2)^{-1}|_{C^\infty_c(\X^{\max})}$ extends as a meromorphic operator valued function of $z \in \C$ with simple poles at $z=-in$, where $n\in \ZZ_{\geq 0}$. 
The residue space at  each pole $z=-in$ is 
\begin{equation}
\label{residuespace-SL}
\left\{\Res_{z=-in} (\mathcal{C}^+-z^2)^{-1}f ; f\in C^\infty_c(\X^{\max})\right\}\,.
\end{equation}
This space is contained in $\mathcal{S}(\X)^*$ and therefore a $\G\G'$-module via $\omega_0$.
As a $\G$-module, the residue space \eqref{residuespace-SL} equals 
$\omega_0(\Theta_{\pi_{\varepsilon,n}^c})(C^\infty_c(\X^{\max}))$, where $\varepsilon\in \{0,1\}$ and $\varepsilon \not\equiv n$. The structure of the residue representations as $\G\G'$-module via $\omega_0$ is then collected in the following theorem.

\begin{thm}
As $\G\G'$-module, $\omega_0(\Theta_{\pi_{\varepsilon,n}^c})(C^\infty_c(\X^{\max}))$ decomposes as follows:
\begin{itemize}
\item If $n=0$, then 
\[
\omega_0(\Theta_{\pi_{1,0}^c})(C^\infty_c(\X^{\max}))= \left(\Dsf^0_-\otimes (\Dsf^0_-)' \right)\oplus \left(\Dsf^0_+\otimes (\Dsf^0_+)'\right)\,,
\]
where $(\Dsf^0_\pm)'$ is the irreducible unitary representation of $\G'$ corresponding to 
$\Dsf^0_\pm$ in Howe's correspondence.
\item If $n\in \ZZ_{>0}$, then
\[
\omega_0(\Theta_{\pi_{\varepsilon,n}^c})(C^\infty_c(\X^{\max}))=
\left(\Dsf^{-n}\otimes (\Dsf^{-n})' \right)\oplus 
\left(\V_n\otimes(\V_n)'\right) \oplus
\left(\Dsf^{n}\otimes (\Dsf^{n})'\right)\,,
\]
where $(\Dsf^{\pm n})'$ is the irreducible unitary representation of $\G'$ corresponding to 
$\Dsf^{\pm n }$ in Howe's correspondence, and $(\V_n)'$ is an admissible quasi-simple representation of $\G'$. If $n=1$, then $\V_1$ is the trivial representation and $(\V_1)'$ is an irreducible unitary representation of $\G'$.
\end{itemize}
\end{thm}

\appendix

\section{The Weil representation}
\label{appendix:Weil representation}

In this appendix we recall the definition of the Weil representation. We follow the approach initiated in \cite{AubertPrzebinda_omega}.

Let $\Wv$ be a finite dimensional real vector space equipped with a non-degenerate symplectic form $\langle{\cdot},{\cdot}\rangle$ and let ${\Sp}(\Wv)$ denote the corresponding symplectic group, with symplectic Lie algebra $\sp(\Wv)$. The metaplectic group 
is the double cover of ${\Sp}(\Wv)$ given by 
\[
\wt{\Sp}(\Wv)=\left\{\wt{g}=(g,\xi)\in \Sp(\Wv) \times \C; \xi^2=\Theta^2(g)\right\}
\]
with group multiplication 
\begin{equation}
\label{multiplication-wtSp}
(g_1,\xi_1)(g_2,\xi_2)=(g_1g_2,\xi_1\xi_2 C(g_1,g_2))\,,
\end{equation}
The $2$-cocyle $C(g_1,g_2)$ appearing in \eqref{multiplication-wtSp} is explicit and can be found in \cite[Proposition 4.13]{AubertPrzebinda_omega}, whereas $\Theta^2$ is defined by 
\[
\Theta^2(g)=\gamma(1)^{2\dim\,(g-1)\Wv-2}\,(\det (g-1:\Wv/\Ker(g-1)\to (g-1)\Wv))^{-1}  \qquad (g\in \Sp(\Wv))\,,
\]
where for every $A\in \GL_n(\R)$ (with $n\geq 1$)
\[
\gamma(\det A)=\frac{e^{\frac{\pi i}{4} \sign(\det A)}}{\sqrt{|\det(A)|} }\,;
\]
see \cite[Definition 4.16 and Remark 4.5]{AubertPrzebinda_omega}.
In particular, $\gamma(1)=e^{\frac{\pi i}{4}}$.

A positive definite compatible complex structure on $(\Wv,\inner{\cdot}{\cdot})$ is an element $J\in \Sp(\Wv)$ such that $J^2=-1$ and the symmetric bilinear form defined on $\Wv$ by $B(w,w')=\inner{Jw}{w'}$ is positive definite. Fix such a $J$. 
For any subspace $\Uv$ of $\Wv$ we normalize the Haar measure $\mu_\Uv$ on $\Uv$ so that the volume of the unit cube with respect to $B$ is equal to $1$. 

Let us fix the unitary character $\chi$ of $\R$ defined by $\chi(r)=e^{2\pi i r}$ and a polarization $\Wv=\X\oplus \Y$. The Weil representation of $\wt{\Sp}(\Wv)$ attached to $\chi$ is defined as the composition of three operators, $\Op$, $\mathcal K$ and $T$, which we now recall. 
$\Op$ is the isomorphism of linear topological vector spaces 
$\Op: \Ss^*(\X\times \X)\to \Hom(\Ss(\X),\Ss^*(\X))$ defined 
by 
\[
\Op(K)v(x)=\int_\X K(x,x')v(x')\,d\mu_\X(x') \qquad (K\in \Ss^*(\X\times \X)\,, v\in \Ss(\X))\,.
\]
The operator
$\mathcal K:\Ss^*(\Wv)\to \Ss^*(\X\times \X)$ is the Weyl transform:
it is the topological isomorphism of linear vector spaces defined for 
$f\in \Ss(\Wv)$ by 
\[
\mathcal K(f)(x,x')=\int_\Yv f(x-x'+y)\chi\left(\frac{1}{2}\langle y, x+x'\rangle\right)\,d\mu_\Y(y).
\]
The operator $T$ embeds $\wt{\Sp}(\Wv)$ into $\Ss^*(\Wv)$ as suitably normalized Gaussian measures. 

An imaginary Gaussian on $(g-1)\W$ is defined by
\[
\chi_{c(g)}(u)=\chi\left(\frac{1}{4}\langle(g+1)(g-1)^{-1}u,u\rangle\right) \qquad (u=(g-1)w,\ w\in\Wv).
\]
(Notice that if $g-1$ is invertible, then $c(g)=(g+1)(g-1)^{-1}$ is the Cayley transform of $g$.)
For $\wt{g}=(g,\xi)\in\wt{\Sp}(\Wv)$ we set $\Theta(\wt{g})=\xi$ and 
define
\[
T(\wt{g})=\Theta(\wt{g})\chi_{c(g)}\mu_{(g-1)\Wv}\,.
\]
Then the Weil representation $(\omega, \L^2(\Xv))$ attached to the character $\chi$ is 
\begin{equation}
\label{omega}
\omega=\Op \circ \mathcal{K} \circ T\,.
\end{equation}
See \cite[Theorem 4.27]{AubertPrzebinda_omega}.
It is a unitary representation of $\wt{\Sp}(\Wv)$ with space of smooth 
vectors equal to $\mathcal{S}(\X)$.

Despite \eqref{omega} defines $\omega(\wt{g})$ for all $\wt{g}\in \wt{\Sp}(\Wv)$, it is not easy to make its right-hand side explicit for arbitrary $\wt{g}\in \wt{\Sp}(\Wv)$. 
As we are going to see, such explicit formulas can be given on certain subgroups of $\wt{\Sp}(\Wv)$. 

The function 
\begin{equation}
\label{character-chi-plus}
\chi_+(\wt{g})=\frac{\Theta(\wt{g})}{|\Theta(\wt{g})|}
\end{equation}
is well defined on the whole metaplectic group $\wt{\Sp}(\Wv)$ and has values in $\Ug_1$. But it is not a character, because $\wt{\Sp}(\Wv)$ does not have any non-trivial unitary character. However, when restricted to specific subgroups, it becomes a character. 
Let $\Wv=\X\oplus \Y$ be any polarization, and let $\M$ be the subgroup of $\Sp(\W)$ preserving $\X$ and $\Y$. Then $\chi_+$ is a character of
the preimage $\wt\M$ of $\M$ in $\wt{\Sp}(\Wv)$. 

For $g\in M$, let $\det(g|_\X)$ denote the determinant of $g$ acting on $\X$. Then
the formula 
\[
\det_\X^{-1/2}(\wt{g})=\chi_+(\wt{g}) |\det_X(g)|^{-1/2}
\]
defines a continuous group homomorphism $\det_\X^{-1/2}: \wt{\M} \to \C^\times$ such that 
$\left(\det_\X^{-1/2}(\wt{g})\right)^2=\det(g|_\X)^{-1}$ for all $\wt{g}\in \wt{\M}$. Moreover, 
\begin{equation}
\label{omegaonM}
\omega(\wt{g})v(x)=\det_\X^{-1/2}(\wt{g}) v(g^{-1}x) \qquad (\wt{g}\in \wt{\M}, \, v\in \mathcal{S}(\X),\, x\in \X)\,.
\end{equation}
See \cite[Proposition 4.28]{AubertPrzebinda_omega}. In particular, if $\wt{g}\in \wt{\M}$ and  
$\det(g|_\X)=1$, then 
\begin{equation}
\label{omegaonM-det1}
\omega(\wt{g})v(x)=v(g^{-1}x) \qquad (v\in \mathcal{S}(\X),\, x\in \X)\,.
\end{equation}

Suppose now that $(\G,\G')=(\Sp_{2n}(\R), \Og_{p,p})$ or $(\Og_{p,p}(\R), \Sp_{2n}(\R))$.
Then both $\G$ and $\G'$ preserve two (different) polarizations of $\W$. Then the restriction of 
$\chi_+$ to each of $\wt\G$ and $\wt\G'$ of is a character. Moreover these restrictions agree on 
the intersection $\wt\G \cap \wt\G'$ (because given by same function). Therefore $\chi_+$ is a character of $\wt\G\wt\G'$. 
Therefore
\begin{equation}\label{twistedrep}
\omega_0(\wt{g})=\chi_+(\wt{g})^{-1} \omega(\wt{g})
\end{equation}
is a representation of $\wt \G \wt \G'$, which is constant on the fibers of the covering. Hence it defines a representation of $\G\G'$ which we denote by the same symbol. 
Thus, for these pairs  $(\G,\G')$, 
we work not with $\omega$ but with $\omega_0$. 

\begin{rem}
If $(\G,\G')=(\Sp_{2n}(\R), \Og_{p,q})$ or $(\G,\G')=(\Og_{p,q},\Sp_{2n}(\R))$ with $p+q$ odd then  
there is no character twisting of $\omega$ allowing to reduce it to a representation on $\G\G'$.
This case includes for instance that of $(\Og_{1}, \Sp_{2n}(\R))$, where $\G\G'=\Sp_{2n}(\R)$. 

Suppose that $(\G,\G')=(\Og_{p,q},\Sp_{2n}(\R))$ with $p\leq q$.
Let $\omega_n$ denote the Weil representation of $\wt{\Sp}_{2n}(\R)$. Then 
\begin{align*}
\omega|_{\wt{\Sp}_{2n}(\R)}&=
\underbrace{\omega_n \otimes \cdots \otimes \omega_n}_{\text{$p$ times}} \otimes
\underbrace{\omega_n^c \otimes \cdots \otimes \omega_n^c}_{\text{$q$ times}}\\
&=\underbrace{(\omega_n \otimes \omega_n^c) \otimes\cdots (\otimes \omega_n \otimes \omega_n^c)}_{\text{$p$ times}} \otimes
\underbrace{\omega_n^c \otimes \cdots \otimes \omega_n^c}_{\text{$q-p$ times}}\,.
\end{align*}
Each tensor product $\omega_n \otimes \omega_n^c$ is a representation of $\wt{\Sp}_{2n}(\R)$ 
constant on the fibers of the metaplectic cover. Hence it gives a representation of $\Sp_{2n}(\R)$.
For the tensor product of $\omega_n^c$, it splits if and only if $q-p$ (i.e. $p+q$) is even. 

As a result, if $p+q$ is even then $\omega|_{\wt{\Sp}_{2n}(\R)}$ splits and the same character 
can be used to twist $\omega|_{\wt\G \wt\G'}$.
\end{rem}

Notice that if $\G$ and $\G'$ do not preserve the same polarization, then we get from \eqref{omegaonM} two different formulas for $\omega_0|_{\G}$ and $\omega_0|_{\G'}$. 
So, if $\G\subset \M$ but $\G'$ is not contained in $\M$, then \eqref{omegaonM} applies to $\G$ but not to $\G'$.
 
For instance, in the case of $(\G,\G')=(\SL_2(\R),\Og_{p,p})$, the formula for $\omega_0$ given in section \ref{section:SL2-Opp} corresponds to \eqref{omegaonM} (and more precisely \eqref{omegaonM-det1} since $\det(g|_\X)=1$ for $g\in \SL_2(\R)$) for a polarization 
of $\W=\X\oplus \Y$ which is preserved by $\SL_2(\R)$ but not by $\Og_{p,p}$. 
Here $\W=M_{2,2p}$ is equipped of the symplectic form 
\[
\inner{w_1}{w_2}=\tr(w_1 w_2^*) \qquad (w_1,w_2\in \Wv)
\]
where $w^*=sw^T j$ for $w\in \Wv$,  
\[
j=\begin{pmatrix}
0 & 1\\
-1 & 0
\end{pmatrix}
\qquad \text{and} \qquad
s=\begin{pmatrix}
0 & 1_p\\
1_p & 0
\end{pmatrix}\,,
\]
and $\X$ and $\Y$ consist respectively of the first two rows and the last two rows of the elements of $\W$. The case $(\G,\G')=(\Og_{1,1},\Sp_2(\R))$ 
is detailed in Appendix \ref{appendix:O11Sp2}.

\section{The dual pair $(\Og_{1,1},\Sp_2(\R))$ in $\Sp_4(\R)$}
\label{appendix:O11Sp2}

Let $\Wv=M_{2,2}(\R)$. For $w\in \Wv$ set $w^*=jw^T s$, where 
\[
j=\begin{pmatrix}
0 & 1\\
-1 & 0
\end{pmatrix}
\qquad \text{and} \qquad
s=\begin{pmatrix}
0 & 1\\
1 & 0
\end{pmatrix}
\]
are as in \eqref{J} and \eqref{element s}, respectively, and $w^T$ denotes the transpose of $w$.
We endow $\Wv$ with the non-degenerate symplectic form 
\begin{equation}
\label{symplectic-form}
\inner{w_1}{w_2}=\tr(w_1 w_2^*) \qquad (w_1,w_2\in \Wv)
\end{equation}
and denote by $\Sp_4(\R)$ the symplectic group of $(\Wv,\inner{\cdot}{\cdot})$.
The actions of $\Og_{1,1}$ and $\Sp_2(\R)$ on $\Wv$ respectively defined by 
\begin{align}\label{convention1}
&h(w)=hw \qquad (h\in\Og_{1,1}, w\in \Wv)\\
\label{convention2}
&g(w)=wg^{-1} \qquad (g\in\Sp_2(\R), w\in \Wv)
\end{align}
embed $\Og_{1,1}$ and $\Sp_2(\R)$ in $\Sp_4(\R)$ as mutually centralizing subgroups. 

Set 
\[
\X=\left\{ \begin{pmatrix}
x_1 & x_2 \\ 0 & 0
\end{pmatrix}; x_1,x_2\in \R\right\} \qquad \text{and}\qquad 
\Y=\left\{ \begin{pmatrix}
0 & 0 \\ y_1 & y_2
\end{pmatrix}; y_1,y_2\in \R\right\}\,.
\]
Then $\Wv=\X\oplus \Y$ is a complete polarization. 
Each element  $h_a=\begin{pmatrix}
a & 0 \\ 0 & a^{-1}
\end{pmatrix} \in \SOg_{1,1}$ preserves $\X$ and $\Y$, and $\det(h_a|_\X)=a^2$. 
Likewise, each element $g\in \Sp_2(\R)$ preserves $\X$ and $\Y$, and 
$\det (g|_\X)=1$. 

Let $\wt{\Sp}_4(\R) \ni \wt{g} \mapsto g \in \Sp_4(\R)$ be the metaplectic covering map and 
let $\wt{\SOg}_{1,1}$ and $\wt{\Sp}_2(\R)$ respectively denote the inverse image of $\SOg_{1,1}$ and $\Sp_2(\R)$ in $\wt{\Sp}_4(\R)$. 
Further, let $\M$ be the subgroup of $\Sp_4(\R)$ consisting of all elements preserving $\X$ and $\Y$, and let $\wt{\M}$ be its inverse image in $\wt{\Sp}_4(\R)$. Hence $\wt{\SOg}_{1,1}\cdot\wt{\Sp}_2(\R) \subseteq \wt{\M}$.
Let $(\omega,\L^2(\X))$ be the Weil representation of $\wt{\Sp}_4(\R)$
(attached to the character $\chi(r)=e^{2\pi i r}$ of $\R$).

By \eqref{omegaonM},
\[
\omega(\wt{g})v(x)=\det_\X^{-1/2}(\wt{g}) v(g^{-1}x) \qquad (\wt{g}\in \wt{\M}, \, v\in \mathcal{S}(\X),\, x\in \X)\,.\
\]
Since $\det (g|_\X)\neq 0$ for $g\in \SOg_{1,1}\cdot\Sp_2(\R)$, then
$\omega|_{\wt{\SOg}_{1,1}\cdot\wt{\Sp_2}(\R)}$ splits and we may choose a section $g\mapsto \wt{g}$ such that, by setting $\omega(g)=\omega(\wt{g})$, we have
\begin{align*}
&\omega(h_a)v(x)=|a|^{-1} v(h_a^{-1}x)=|a|^{-1} v(a^{-1}x) \qquad (a\in\R^\times, \, v\in \mathcal{S}(\X),\, x\in \X),\\
&\omega(g)v(x)=v(g^{-1}x)=v(xg) \qquad (g\in \Sp_2(\R), \, v\in \mathcal{S}(\X),\, x\in \X).
\end{align*}
(Observe that the right-hand sides agree on $\{\pm 1\}=\SOg_{1,1} \cap \Sp_2(\R)$.)  
The argument above applies to $\SOg_{1,1}$ but not to 
$\Og_{1,1}$ as $\Og_{1,1}=\SOg_{1,1}\sqcup s\SOg_{1,1}$ and 
$s$ does not preserve $\X$ and $\Y$. 
(Notice that $\Og_{1,1}$ preserve the polarization of $\W$ given by first and second columns.)
Nevertheless, it enough to understand $\omega(\wt{s})$. 
For this, we use the explicit definition of the Weil representation as given in 
\cite[Theorem 4.27]{AubertPrzebinda_omega}. 

Since by \eqref{convention1} $(s-1)(w)=(s-1)w$ we see that $\det_{\Wv}(s-1)=0$. Hence $\Wv\supsetneqq (s-1)\Wv$. This is going to force us to make some work. 

\begin{lem}
Let $\wt{s}\in \wt{\Og}_{1,1}\subseteq \wt{\Sp}(\Wv)$ be an inverse image of $s$.
Then, in the notation of Appendix \ref{appendix:Weil representation},
\[
\Theta(\wt{s})=\pm \frac{1}{2}\,,\qquad
T(\wt{s})=\pm \frac{1}{2} \mu_{(s-1)\Wv}\,, \qquad
\mathcal{K}(T(\wt{s}))(x,x')=\pm \chi(x' jx^T).
\]
Thus 
\begin{equation}
\label{omegawts}
\omega(\wt{s})v(x)=\pm \int_\X \chi(x'jx^T)v(x') \, dx'\qquad (v\in \mathcal{S}(\X))\,,
\end{equation}
where $dx=dx_1\,dx_2$ is the Lebsegue measure on $\X=\R^2$.
\end{lem}
\begin{proof}
Define $J(w)=-swj$. Then $J^2=-1$ and 
\begin{equation}
\label{innerproduct-1}
\inner{J w_1}{w_2}=\tr(-sw_1jjw_2^Ts)=\tr(w_1w_2^T) \qquad
(w_1,w_2\in \Wv)\,.
\end{equation}
is a positive definite symmetric bilinear form on $\Wv$. 
Hence $J\in \Sp_4(\R)$ is a positive definite compatible complex structure on $\Wv$. 

Let $L=J^{-1}(s-1)$. Using the convention \eqref{convention1}, explicitly we have  
\[
L(w)=-J(s-1)(w)=-s(s-1)wj=(s-1)wj\,.
\]
Hence
\[
LW=(s-1)W=\left\{ \begin{pmatrix}
x\\-x 
\end{pmatrix}; x\in M_{1,2}(\R)\right\}\,.
\]
Furthermore, for $u\in LW$, we have $L(u)=2uj$.
Hence, 
\[
4=\det\left(L|_{LW}\right)=\det\left(s-1:\Wv/\Ker(s-1) \to 
\Im(s-1)\right).
\]
Since $\gamma(1)=e^{\frac{\pi i}{4}}$, we obtain
\[
\Theta^2(s)=\left(e^{\frac{\pi i}{4}}\right)^{2\dim LW} 4^{-1}=4^{-1}\,.
\]
Hence $\Theta(\wt{s})=\pm \frac{1}{2}$.

For $T(\wt{s})$, we need to compute $\chi_{c(s)}$. 
Since 
$(s-1)\begin{pmatrix}
0 \\ x
\end{pmatrix}=\begin{pmatrix}
x\\-x
\end{pmatrix}$, it follows that
$(s-1)^{-1}\begin{pmatrix}
x \\ -x
\end{pmatrix}=\begin{pmatrix}
0\\x
\end{pmatrix}$. Hence
\[
(s+1)(s-1)^{-1}\begin{pmatrix}
x \\ -x
\end{pmatrix}=(s+1)\begin{pmatrix}
0 \\ x
\end{pmatrix}=\begin{pmatrix}
x \\ x
\end{pmatrix}\,.
\]
Thus, for $\begin{pmatrix}
x \\ -x
\end{pmatrix}\in (s-1)\Wv$,  
\begin{align*}
\langle(s+1)(s-1)^{-1}&\begin{pmatrix}
x \\ -x
\end{pmatrix},{\begin{pmatrix}
x \\ -x
\end{pmatrix}}\rangle=
\langle\begin{pmatrix}
x \\ -x
\end{pmatrix},{\begin{pmatrix}
x \\ -x
\end{pmatrix}}\rangle=
\tr\left(\begin{pmatrix}
x \\ -x
\end{pmatrix}j\begin{pmatrix}
x \\ -x
\end{pmatrix}^Ts\right)\\
&=
\tr\left(\begin{pmatrix}
xjx^T & -xjx^T\\
xjx^T & -xjx^T
\end{pmatrix}
\begin{pmatrix}
0 & 1\\
1 & 0
\end{pmatrix}\right)
=
\tr\left(\begin{pmatrix}
-xjx^T & xjx^T\\
-xjx^T & xjx^T
\end{pmatrix}\right)=0\,.
\end{align*}
Therefore, $\chi_{c(s)}=1$ and 
\[
T(\wt{s})=\Theta(\wt{s})\chi_{c(s)} \mu_{(s-1)\Wv}= \pm \frac{1}{2} \mu_{(s-1)\Wv}.
\]
We now determine the Haar measures on $\X$, $\Y$ and $(s-1)\W$ with the normalizations 
fixed in Appendix \ref{appendix:Weil representation}.
Notice that, by \eqref{innerproduct-1}, the restriction of $B$ to $\X$ is 
\[
B\left(\begin{pmatrix}
x\\0
\end{pmatrix}, \begin{pmatrix}
x'\\0
\end{pmatrix}\right)=\tr\left( \begin{pmatrix}
x\\0
\end{pmatrix} \begin{pmatrix}
{x'}^T & 0
\end{pmatrix}\right)=x{x'}^T\,.
\]
The unit cube in $\X=\R^2$ is $[0,1]^2$. Thus, in the fixed normalizations, $d\mu_\X=dx=dx_1\,dx_2$. Likewise, $d\mu_\Y=dy_1\, dy_2$.

A Haar measure $\mu_{(s-1)\Wv}$ on $(s-1)\Wv$ is a constant multiple of the pullback of the Lebesgue measure $\lambda$ on $\X=\R^2$ via the isomorphism $\alpha: (s-1)\Wv\ni \begin{pmatrix}
x\\-x
\end{pmatrix} \mapsto x\in\R^2$. Hence, as a measure on $\Wv=\X\oplus \Y$, 
we have $\mu_{(s-1)\Wv}(x,y)=C\delta(x+y)\,dx\, dy$. The constant $C\geq 0$ is fixed by the condition that the measure of the unit cube with respect to the restriction to $(s-1)\Wv$ of the inner product $B(w_1,w_2)=\inner{Jw_1}{w_2}$ is one. By \eqref{innerproduct-1},
\[
B\left(\begin{pmatrix}
x\\-x
\end{pmatrix}, \begin{pmatrix}
x'\\-x'
\end{pmatrix}\right)=\tr\left( \begin{pmatrix}
x\\-x
\end{pmatrix} \begin{pmatrix}
{x'}^T & -{x'}^T
\end{pmatrix}\right)=2x{x'}^T\,.
\]
The unit cube in $(s-1)\Wv$ with respect to $B$ is mapped by $\alpha$ into 
$[0,\frac{1}{\sqrt{2}}]^2$, with Lebsegue measure $\frac{1}{2}$. Hence $C=2$ and 
\[
\mu_{(s-1)\Wv}(x,y)=
2\delta(x+y)\,dx\, dy \qquad (x\in \X, y\in \Y).
\]
It follows that 
\begin{align*}
\mathcal{K}(T(\wt{s}))(x,x')
&=\pm \frac{1}{2} \int_\Y \mu_{(s-1)\Wv} \begin{pmatrix}
x-x'\\y
\end{pmatrix} \chi\left( \frac{1}{2}
\langle\begin{pmatrix}
0 \\ y
\end{pmatrix},{\begin{pmatrix}
x+x' \\ 0
\end{pmatrix}}\rangle \right) \;dy\\
&=\pm \frac{1}{2} \int_\Y 
\delta(x-x'+y)
\chi\left( \frac{1}{2}
\inner{\begin{pmatrix}
0 \\ y
\end{pmatrix}}{\begin{pmatrix}
x+x' \\ 0
\end{pmatrix}} \right)\; dy\\
&=\pm \chi\left( \frac{1}{2}
\inner{\begin{pmatrix}
0 \\ x-x'
\end{pmatrix}}{\begin{pmatrix}
x+x' \\ 0
\end{pmatrix}}
\right)\,.
\end{align*}
By \eqref{symplectic-form},
\begin{align*}
\langle\begin{pmatrix}
0 \\ x-x'
\end{pmatrix},&{\begin{pmatrix}
x+x' \\ 0
\end{pmatrix}}\rangle
=\tr \left( 
\begin{pmatrix}
0 \\ x-x'
\end{pmatrix}j\begin{pmatrix}
x^T+{x'}^T & 0
\end{pmatrix} s\right)
=\tr \left( 
\begin{pmatrix}
0 \\ x-x'
\end{pmatrix}j\begin{pmatrix}
0 & x^T+{x'}^T 
\end{pmatrix} \right)\\
&=\tr \left( 
\begin{pmatrix}
0 \\ xj-x'j
\end{pmatrix}
\begin{pmatrix}
0 & x^T+{x'}^T 
\end{pmatrix} \right)
=\tr \left( 
\begin{pmatrix}
0 & 0 \\ 0 & (xj-x'j)(x^T+{x'}^T)
\end{pmatrix}	\right)\\
&=(xj-x'j)(x^T+{x'}^T)=2x'jx^T
\end{align*}
since $xjx^T=x'j{x'}^T=0$. The formulas for $\mathcal{K}(T(\wt{s}))$ and 
$\omega(\wt{s})$ therefore follow.
\end{proof}

\biblio

\begin{thebibliography}{Wal88b}

\bibitem[AP14]{AubertPrzebinda_omega}
A.-M. Aubert and T.~Przebinda.
\newblock {A reverse engineering approach to the Weil Representation}.
\newblock {\em {Central Eur. J. Math.}}, {12}:{1500--1585}, {2014}.

\bibitem[Arn76]{Arnal76}
D.~Arnal.
\newblock Symmetric nonself-adjoint operators in an enveloping algebra.
\newblock {\em J. Funct. Anal.}, 21:432--447, 1976.

\bibitem[Dix69]{Dixmier1969-Hilbert}
J.~Dixmier.
\newblock {\em Les alg\`ebres d'op\'{e}rateurs dans l'espace hilbertien
  (alg\`ebres de von {N}eumann)}.
\newblock Cahiers Scientifiques, Fasc. XXV. Gauthier-Villars \'{E}diteur,
  Paris, 1969.

\bibitem[DM78]{DixmierMalliavin78}
J.~Dixmier and P.~Malliavin.
\newblock Factorisations de fonctions et de vecteurs ind\'{e}finiment
  diff\'{e}rentiables.
\newblock {\em Bull. Sci. Math. (2)}, 102(4):307--330, 1978.

\bibitem[DZ19]{DZ19}
S.~Dyatlov and M.~Zworski.
\newblock {\em Mathematical theory of scattering resonances}, volume 200 of
  {\em Graduate Studies in Mathematics}.
\newblock American Mathematical Society, Providence, RI, 2019.

\bibitem[Eps26]{Epstein}
P.~Epstein.
\newblock Schroedinger's quantum theory and the stark effect.
\newblock {\em Nature}, 118:444--445, 1926.

\bibitem[F{\"{u}}h05]{Fuhr}
H.~F{\"{u}}hr.
\newblock {\em Abstract harmonic analysis of continuous wavelet transforms},
  volume 1863 of {\em Lecture Notes in Mathematics}.
\newblock Springer-Verlag, Berlin, 2005.

\bibitem[Har07]{Harrell}
E.~M. Harrell, II.
\newblock Perturbation theory and atomic resonances since {S}chr\"{o}dinger's
  time.
\newblock In {\em Spectral theory and mathematical physics: a {F}estschrift in
  honor of {B}arry {S}imon's 60th birthday}, volume~76 of {\em Proc. Sympos.
  Pure Math.}, pages 227--248. Amer. Math. Soc., Providence, RI, 2007.

\bibitem[H{\"o}r83]{Hormander}
L.~H{\"o}rmander.
\newblock {\em The analysis of linear partial differential operators. {I}},
  volume 256 of {\em Grundlehren der Mathematischen Wissenschaften [Fundamental
  Principles of Mathematical Sciences]}.
\newblock Springer-Verlag, Berlin, 1983.
\newblock Distribution theory and Fourier analysis.

\bibitem[How79]{HoweL2}
R.~Howe.
\newblock {L}$\null^2$-duality for stable reductive dual pairs.
\newblock preprint, 1979.

\bibitem[How89]{HoweRemarks}
R.~Howe.
\newblock {Remarks on Classical Invariant Theory}.
\newblock {\em {Trans. Amer. Math. Soc.}}, {313}:{539--570}, {1989}.

\bibitem[HP09]{HP09}
J.~Hilgert and A.~Pasquale.
\newblock {Resonances and residue operators for symmetric spaces of rank one}.
\newblock {\em {J. Math. Pures Appl. (9)}}, {91}:{495--507}, {2009}.

\bibitem[HT92]{HoweTan}
R.~Howe and E.-C. Tan.
\newblock {\em Nonabelian harmonic analysis}.
\newblock Universitext. Springer-Verlag, New York, 1992.
\newblock Applications of ${{\rm{S}}L}(2,{{\bf{R}}})$.

\bibitem[HT93]{HT93}
R.~Howe and E.-C. Tan.
\newblock Homogeneous functions on light cones: the infinitesimal structure of
  some degenerate principal series representations.
\newblock {\em Bull. Amer. Math. Soc. (N.S.)}, 28(1):1--74, 1993.

\bibitem[HU91]{HoweUmeda}
R.~Howe and T.~Umeda.
\newblock The {C}apelli identity, the double commutant theorem, and
  multiplicity-free actions.
\newblock {\em Math. Ann.}, 290(3):565--619, 1991.

\bibitem[Ito05]{Itoh05}
M.~Itoh.
\newblock Capelli identities for reductive dual pairs.
\newblock {\em Adv. Math.}, 194(2):345--397, 2005.

\bibitem[Kna86]{knappLie2}
A.~W. Knapp.
\newblock {\em Representation theory of semisimple groups}, volume~36 of {\em
  Princeton Mathematical Series}.
\newblock Princeton University Press, Princeton, NJ, 1986.

\bibitem[Lan85]{LangSL2R}
S.~Lang.
\newblock {\em {${\rm SL}_2({\bf R})$}}, volume 105 of {\em Graduate Texts in
  Mathematics}.
\newblock Springer-Verlag, New York, 1985.

\bibitem[Mac76]{MackeyUnitaryBook}
G.~W. Mackey.
\newblock {\em The theory of unitary group representations}.
\newblock Chicago Lectures in Mathematics. University of Chicago Press,
  Chicago, Ill.-London, 1976.

\bibitem[MW00]{MW00}
R.~J. Miatello and C.~E. Will.
\newblock {The residues of the resolvent on Damek-Ricci spaces}.
\newblock {\em {Proc. Amer. Math. Soc.}}, {128}:{1221--1229}, {2000}.

\bibitem[Nus64]{Nussbaum-Duke64}
A.~E. Nussbaum.
\newblock Reduction theory for unbounded closed operators in {H}ilbert space.
\newblock {\em Duke Math. J.}, 31:33--44, 1964.

\bibitem[Pou72]{Poulsen72}
N.~S. Poulsen.
\newblock On {$C^{\infty }$}-vectors and intertwining bilinear forms for
  representations of {L}ie groups.
\newblock {\em J. Funct. Anal.}, 9:87--120, 1972.

\bibitem[Prz93]{PrzebindaUnitary}
T.~Przebinda.
\newblock Characters, dual pairs, and unitary representations.
\newblock {\em Duke Math. J.}, 69(3):547--592, 1993.

\bibitem[Prz96]{PrzebindaInfinitesimal}
T.~Przebinda.
\newblock {The duality correspondence of infinitesimal characters}.
\newblock {\em { Coll. Math}}, {70}:{93--102}, {1996}.

\bibitem[Sch90]{Schmuedgen90}
K.~Schm\"{u}dgen.
\newblock {\em Unbounded operator algebras and representation theory},
  volume~37 of {\em Operator Theory: Advances and Applications}.
\newblock Birkh\"{a}user Verlag, Basel, 1990.

\bibitem[Seg59]{Segal59}
I.~E. Segal.
\newblock A theorem on the measurability of group-invariant operators.
\newblock {\em Duke Math. J.}, 26:549--552, 1959.

\bibitem[SW71]{SteinWeiss}
E.~M. Stein and G.~Weiss.
\newblock {\em Introduction to {F}ourier analysis on {E}uclidean spaces}.
\newblock Princeton University Press, Princeton, N.J., 1971.

\bibitem[Wal88a]{WallachI}
N.~Wallach.
\newblock {\em {Real Reductive Groups I}}.
\newblock {Academic Press}, {1988}.

\bibitem[Wal88b]{WallachII}
N.~Wallach.
\newblock {\em {Real Reductive Groups II}}.
\newblock {Academic Press}, {1988}.

\bibitem[Wat95]{WatsonBessel}
G.~N. Watson.
\newblock {\em A treatise on the theory of {B}essel functions}.
\newblock Cambridge Mathematical Library. Cambridge University Press,
  Cambridge, 1995.

\end{thebibliography}
\end{document}